\documentclass[a4paper,12pt, abstract=true]{scrartcl}
\usepackage[ansinew]{inputenc}
\usepackage[english]{babel}
\usepackage{latexsym,amssymb,amsfonts,amsmath,bbm}%amsmath
\usepackage{theorem}
\usepackage{xcolor}
\usepackage{graphicx}
\usepackage{esint}
\usepackage{stmaryrd}
\usepackage[normalem]{ulem} % fuer \sout u.s.w.
\usepackage{lipsum}
\usepackage{epstopdf}
\usepackage{algorithmic}
\ifpdf
  \DeclareGraphicsExtensions{.eps,.pdf,.png,.jpg}
\else
  \DeclareGraphicsExtensions{.eps}
\fi

\numberwithin{equation}{section}
\theoremstyle{plain}
\newtheorem{theorem}{Theorem}[section]
\newtheorem{lemma}[theorem]{Lemma}
\newtheorem{proposition}[theorem]{Proposition}
\newtheorem{corollary}[theorem]{Corollary}
\newcommand{\proofendsign}{$\Box$} % \rule{2mm}{2mm}
\newtheorem{remark}[theorem]{Remark}
\newtheorem{example}[theorem]{Example}

\newenvironment{proof}{{\noindent \textbf{Proof} }}
 {{\hspace*{\fill}\proofendsign\par\bigskip}}
\newcommand{\Flinks}{F^{\leftarrow}}

\newcommand{\NNN}{\mathbb{N}}

\newcommand{\R}{\mathbb{R}}
\newcommand{\RRR}{\mathbb{R}}

\newcommand{\FFF}{\mathbb{F}}

\newcommand{\EEE}{\mathbb{E}}

\newcommand{\pr}{\mathbb{P}}
\newcommand{\MP}{\mathbb{P}}
\newcommand{\MQ}{\mathbb{Q}}

\newcommand{\ex}{\mathbb{E}}
\newcommand{\vari}{\mathbb{V}\textrm{ar}}

\newcommand{\eins}{\mathbbm{1}}

\newcommand{\cB}{\mathcal B}

\newcommand{\cF}{\mathcal F}

\newcommand{\cH}{\mathcal H}

\newcommand{\cM}{\mathcal M}
\newcommand{\cP}{\mathcal P}

\newcommand{\cR}{\mathcal R}

\newcommand{\OFP}{(\Omega,{\cal F},\pr)}

\def\bcswitch{\left\{\renewcommand{\arraystretch}{1.2}\begin{array}{c@{,~}c}}

\def\ecswitch{\end{array}\right.}

\begin{document}
\title{Nonasymptotic upper estimates for errors of the sample average approximation method to solve risk averse stochastic programs}
\author{
Volker Krätschmer
\footnote{Faculty of Mathematics, University of Duisburg--Essen, { volker.kraetschmer@uni-due.de}}
}
\date{~}
\maketitle

\begin{abstract}~We study statistical properties of the optimal value of the Sample Average Approximation. The focus is on the tail function of the absolute error induced by the Sample Average Approximation, deriving upper estimates of its outcomes dependent on the sample size. The estimates allow to conclude immediately convergence rates for the optimal value of the Sample Average Approximation. As a crucial point the investigations are based on a new type of conditions from the theory of empirical processes which do not rely on pathwise analytical properties of the goal functions. In particular, continuity in the parameter is not imposed in advance as often in the literature on the Sample Average Approximation method. It is also shown that the new condition is satisfied if the paths of the goal functions are H\"older continuous so that the main results carry over in this case. Moreover, the main results are applied to goal functions whose paths are piecewise H\"older continuous as e.g. in two stage mixed-integer programs. The main results are shown for classical risk neutral stochastic programs, but we also demonstrate how to apply them to the sample average approximation of risk averse stochastic programs. In this respect we consider stochastic programs expressed in terms of mean upper semideviations and divergence risk measures.
\end{abstract}

\textbf{keywords:}
Risk averse stochastic program, Sample Average Approximation, mean upper semideviations, divergence risk measures, Talagrand's inequalities, covering numbers, VC-subgraph classes.
%\begin{MSCcodes}
%90C15, 60E15, 60B12
%\end{MSCcodes}

%[class=MSC]
%\kwd[Primary]{~}
%\kwd[; secondary]{~}
%\end{keyword}
%\begin{keyword} 

%\textbf{Keywords:} Risk averse stochastic program, Sample Average Approximation, divergence risk measures, normal integrands, covering numbers, bracketing numbers, VC-subgraph classes.
\section{Introduction}
Consider a classical risk neutral stochastic program 
\begin{equation}
\label{optimization risk neutral}
\inf_{\theta\in\Theta}\EEE\big[G(\theta,Z)\big],
\end{equation}
where $\Theta$ denotes a compact subset of $\RRR^{m}$, whereas $Z$ stands for a $d$-dimensional random vector with distribution $\MP^{Z}$. In general the parameterized distribution of the goal function $G$ is unknown, but some information is available by i.i.d. samples. Using this information, a general device to solve approximately  problem (\ref{optimization risk neutral}) is provided by the so-called \textit{Sample Average Approximation} (\textit{SAA}) (see \cite{ShapiroEtAl}). For explanation, let us consider a sequence $(Z_{j})_{j\in\NNN}$ of independent $d$-dimensional random vectors on some fixed atomless complete probability space $\OFP$ which are identically distributed as the $d$-dimensional random vector $Z$.
%we shall assume
%\begin{enumerate}
%\item [(A 3)] $\{\sqrt{F_{\theta}\cdot (1- F_{\theta})}\mid \theta\in\Theta\}$ is uniformly $\mu$-integrable.
%\item [(A 4)] There exist $\varepsilon$ and $\theta^{*}\in\Theta$ such that
%$$
%\EEE\left[\exp\left(\frac{C(Z)^{2}}{\varepsilon_{1}}\right)\right] < \infty~\quad\mbox{and}\quad 
%\EEE\left[G(\theta^{*},Z)^{2}\right] < \infty.
%$$
%\end{enumerate}
Let us set
$$
\hat{F}_{n,\theta}(t) := \frac{1}{n}~\sum_{j=1}^{n}\eins_{]-\infty,t]}\big(G(\theta,Z_{j})\big)
$$
to define the empirical distribution function $\hat{F}_{n,\theta}$ of $G(\theta,Z)$ based on the i.i.d. sample 
$(Z_{1},\cdots,Z_{n})$. Then the SAA method approximates the genuine optimization problem (\ref{optimization risk neutral}) by the following one
\begin{equation}
\label{SAAriskneutral}
\inf_{\theta\in\Theta}\int_{\RRR}t ~d\hat{F}_{n,\theta}(t) = \inf_{\theta\in\Theta}\frac{1}{n}~\sum_{j=1}^{n}G(\theta,Z_{j})\quad(n\in\NNN).
\end{equation}
The optimal values depend on the sample size and the realization of the samples of $Z$. Their asymptotic behaviour with increasing sample size, also known as the first order asymptotics of (\ref{optimization risk neutral}), is well-known. More precisely, 
the sequence of optimal values of the approximated optimization problem converges 
$\MP$-a.s. to the optimal value of the genuine stochastic program. Moreover, if $G$ is Lipschitz continuous in $\theta$, then the stochastic sequence 
$$
\left(\sqrt{n}\Big[\inf_{\theta\in\Theta}\int_{\RRR}t ~d\hat{F}_{n,\theta}(t) - \inf_{\theta\in\Theta}\EEE\big[G(\theta,Z)\big]\Big]\right)_{n\in\NNN}
$$
is asymptotically normally distributed. In \cite{EichhornRoemisch2007} asymptotic distributions of this stochastic sequence have also be found for stochastic mixed integer programs, where typically the objectives are not continuous in the parameter. For these results, and more on asymptotics of the SAA method the reader may consult the monograph \cite{ShapiroEtAl}, and in addition the contributions \cite{Pflug1999}, \cite{Roemisch2003}. 
\medskip

In several fields like finance, insurance or microeconomics, the assumption of risk neutral decision makers are considered to be too idealistic. Instead there it is preferred to study the behaviour of actors with a more cautious attitude, known as risk aversion. In this view the optimization problem (\ref{optimization risk neutral}) should be replaced with a \textit{risk averse} stochastic program, i.e. an optimization problem
\begin{equation}
\label{optimization risk averse}
\inf_{\theta\in\Theta}\rho\big(G(\theta,Z)\big),
\end{equation}
where $\rho$ stands for a functional which is nondecreasing w.r.t. the increasing convex order. A general class of functionals fulfilling this requirement is built by the so called 
\textit{distribution-invariant convex risk measures} (see e.g. \cite{FoellmerSchied2011}, \cite{ShapiroEtAl}). They play an important role as building blocks in quantitative risk management (see \cite{McNeilEmbrechtsFrey2005}, \cite{PflugRoemisch2007}, \cite{Rueschendorf2013}), and they have been suggested as a systematic approach for calculations of insurance premia
(cf. \cite{Kaasetal2008}). Distribution-invariance, sometimes also called \textit{law-invariance}, denotes the property that a functional $\rho$ has the same outcome for random variables with identical distribution. Hence, a distribution-invariant convex risk measure $\rho$ may be associated with a functional $\cR_{\rho}$ on sets of distribution functions. In this case (\ref{optimization risk averse}) reads as follows
$$
\inf_{\theta\in\Theta}\cR_{\rho}(F_{\theta}),
$$
where $F_{\theta}$ is the distribution function of $G(\theta,Z)$. Then we may modify the SAA method by turning over to the \textit{empirical counterpart} of the stochastic program, i.e.
\begin{equation}
\label{SAA risk averse}
\inf_{\theta\in\Theta}\cR_{\rho}(\hat{F}_{n,\theta})\quad(n\in\NNN).
\end{equation}
For ease of reference, with a slight abuse of meaning, we keep the name \textit{Sample Average Approximation}.
\smallskip

It is already known that under rather general conditions on the mapping $G$ we have 
$$
\inf_{\theta\in\Theta}\cR_{\rho}\big(\hat{F}_{n,\theta}\big)\to \inf_{\theta\in\Theta}\cR_{\rho}\big(F_{\theta}\big)\quad\MP-\mbox{a.s.}
$$
%for every risk functional $\cR_{\rho_{g}}$ associated with a concave distortion risk measure $\rho_{g}$ 
(see \cite{Shapiro2013a}). Occasionally, there also exist some contributions on the asymptotic distributions of the optimal values. In \cite{DentchevaEtAl2017} the authors consider functionals $\rho$ of composite form enclosing mean upper semideviations of order $p > 1$. Distribution invariant coherent risk measure based on finitely discrete Kusuoka representations are subjects in \cite{GuiguesKraetschmerShapiro2018}. The results in both references rely on Lipschitz continuity of the objective in the parameter. Independently of this paper asymptotic distributions have been developped for SAA under absolute semideviation and divergence measures in \cite{Kraetschmer2023}. There the investigations do not require the objectives to satisfy certain analytical properties in the parameter.
\medskip 

The subject of this paper is to look at deviation probabilities 
\begin{equation}
\label{Kreuz}
\MP\Big(\Big\{\big|\inf_{\theta\in\Theta}\cR_{\rho}\big(\hat{F}_{n,\theta}\big) - \inf_{\theta\in\Theta}\cR_{\rho}\big(F_{\theta}\big)\big|\geq\varepsilon\Big\}\Big)\quad(n\in\NNN,\varepsilon > 0)
\end{equation}
dependent on the sample size $n$. Such error estimates might be interesting to identify possible convergence rates for the optimal values of the SAA method. Also from a practical viewpoint they might give some hints for which sample sizes the SAA method provides sufficiently satisfying approximations. In the risk neutral case the authors in \cite{GuiguesEtAl2017} deal with upper estimates for the deviation probabilities if the objective $G$ is convex in the parameter, and $G(\theta,Z_{1}) - \EEE[G(\theta,Z_{1)}]$ has subgaussian distribution with upper bound of the variance factor independent of $\theta$. 
Very recently, the issue of deviation probabilities has been addressed for the risk averse stochastic programs in \cite{BartlTangpi2020}, where $G(\cdot,z)$ is assumed to be linear for $z\in\RRR^{d}$. Our contribution is to investigate error estimates for more general goal functions than in \cite{GuiguesEtAl2017} and \cite{BartlTangpi2020}. Furthermore we provide explicit bounds instead of using unspecified universal constants as in \cite{BartlTangpi2020}. 
%however restricting ourselves to upper semideviations and divergence risk measures in order to keep the paper more comprehensive.
\medskip

The paper is organized as follows. We shall start with a general exponential bound for the deviation probabilities 
$$
\MP\Big(\Big\{\Big|\inf\limits_{\theta\in\Theta}~\frac{1}{n}~\sum_{j=1}^{n}G(\theta,Z_{j}) - \inf\limits_{\theta\in\Theta}\EEE\big[G(\theta,Z_{1})\big]\Big|\geq\varepsilon\Big\}\Big)\quad(n\in\NNN, \varepsilon > 0)
$$
in the case of classical risk neutral stochastic programs. The point is that we may extend this result to deviation probabilities if the SAA method is applied to risk averse stochastic programs. In Section \ref{upper semideviations} this will be demonstrated in the case that stochastic programs are expressed in terms of mean upper semideviations, whereas in Section \ref{generalized divergence risk measures} the application to stochastic programs under divergence risk measures is considered. We always find exponential bounds for the deviation probabilities which as an immediate by product give convergence rates for the SAA method in the different contexts. In particular, $\sqrt{n}$-consistency will turn out to be an easy consequence.
Finally Section \ref{Beweise} gathers proof of results from the previous sections.
\par
The essential new ingredient of our results is to replace analytic conditions on the paths $G(\cdot,z)$ with requirements which intuitively make the family $\{G(\theta,Z)\mid\theta\in\Theta\}$ of random variables small in some certain sense. Fortunately, the respective invoked conditions are satisfied if the paths $G(\cdot,z)$ are H\"older continuous. We shall also show that we may utilize our results to study the SAA method for stochastic programs, where the paths $G(\cdot,z)$ are piecewise H\"older continuous but not necessarily continuous. Value functions of two stage mixed-integer programs are typical examples for goal functions of such a kind. 
%%%%%%%%%%%%%%%%%%%%%%%%%%%%%%%%%%%%%%%%%%%%%%%%%%%%%%%%%%%%%%%%%%%%%%%%%%%%%%%%%%%%%%%%%%%%%%%%%%%%%%%%%%%%%%%%%%%%%%%%%%%%%%%%%%%%%%%%%%%%%%%%%%%%%%%%%%%%%%%%%%%%%%%%%%%%%%%%%%%%%%%%%%%%%%%%%%%%%%%%%%%%%%%%%%%%%%%%%%%%%%%%%%%%%%%%%%%%%%%%%%%%%%%%%%%%%%%%%%%%%%%%%%%%%%%%%%%%%%%%%%%%%%%%
\section{Error estimates in the risk neutral case}
\label{error risk neutral}
In this section we study the SAA \eqref{SAAriskneutral} associated with the risk neutral stochastic program \eqref{optimization risk neutral}. 
We shall restrict ourselves to mappings $G$ which satisfy the following properties.
\begin{itemize}
\item[(A 1)] $G(\theta,\cdot)$ is Borel measurable for every $\theta\in\Theta$. 
\item[(A 2)] 
There is some strictly positive $\MP^{Z}$-integrable mapping $\xi:\RRR^{d}\rightarrow\RRR$ such that 
$$
\sup\limits_{\theta\in\Theta}|G(\theta,z)|\leq\xi(z)\quad\mbox{for}~z\in\RRR^{d}.
$$ 
\end{itemize}
Note that under these assumptions the optimization problems \eqref{optimization risk neutral} and \eqref{SAAriskneutral} are well defined with finite optimal values.
\par 
The subject of this section is to investigate 
\begin{equation}
\label{expected error}
\EEE\Big[\Big|\inf\limits_{\theta\in\Theta}~\frac{1}{n}~\sum_{j=1}^{n}G(\theta,Z_{j}) - \inf\limits_{\theta\in\Theta}\EEE\big[G(\theta,Z_{1})\big]\Big|\Big],
\end{equation}
and the probabilities
\begin{equation}
\label{deviations}
\MP\Big(\Big\{\Big|\inf\limits_{\theta\in\Theta}~\frac{1}{n}~\sum_{j=1}^{n}G(\theta,Z_{j}) - \inf\limits_{\theta\in\Theta}\EEE\big[G(\theta,Z_{1})\big]\Big|\geq\varepsilon\Big\}\Big)\quad(n\in\NNN, \varepsilon > 0).
\end{equation}
The aim is to find explicit bounds in terms of the sample sizes $n$. In order to avoid subtleties of measurability we additionally assume
\begin{itemize}
\item [(A 3)] There exist some at most countable subset $\overline{\Theta}\subseteq\Theta$ and $(\MP^{Z})^{n}$-null sets $N_{n}$ $(n\in\NNN)$ such that 
$$
\inf_{\vartheta\in\overline{\Theta}}\big|\EEE[G(\vartheta,Z_{1})] - \EEE[G(\theta,Z_{1})]\big|= \inf_{\vartheta\in\overline{\Theta}}\max_{j\in\{1,\ldots,n\}}\big|G(\theta,z_{j}) - G(\vartheta,z_{j})\big| = 0
$$
for $n\in\NNN,\theta\in\Theta$ and $(z_{1},\ldots,z_{n})\in\RRR^{d n}\setminus N_{n}$.
\end{itemize}
By assumption (A 3) with at most countable subset $\overline{\Theta}\subseteq\Theta$ we have  
$$
\inf_{\theta\in\Theta }\frac{1}{n}\sum_{j=1}^{n}G(\theta,Z_{j}) = \inf_{\theta\in\overline{\Theta}}\frac{1}{n}\sum_{j=1}^{n}G(\theta,Z_{j})~\MP-\mbox{a.s.},\quad\inf_{\theta\in\Theta}\EEE[G(\theta,Z_{1})] = \inf_{\theta\in\overline{\Theta}}\EEE[G(\theta,Z_{1})].
$$
Hence the optimal value of the SAA \eqref{SAAriskneutral} is a random variable on $\OFP$ due to the assumed completeness of this probability space. Moreover, the desired upper estimations of \eqref{expected error} and \eqref{deviations} may be derived by upper estimations of 
\begin{equation}
\label{expected error II}
\EEE\Big[\sup\limits_{\theta\in\overline{\Theta}}\Big|~\frac{1}{n}~\sum_{j=1}^{n}G(\theta,Z_{j}) - \EEE\big[G(\theta,Z_{1})\big]\Big|\Big],
\end{equation}
and 
\begin{equation}
\label{deviations II}
\MP\Big(\Big\{\sup\limits_{\theta\in\overline{\Theta}}\Big|~\frac{1}{n}~\sum_{j=1}^{n}G(\theta,Z_{j}) - \EEE\big[G(\theta,Z_{1})\big]\Big|\geq\varepsilon\Big\}\Big)\quad(n\in\NNN, \varepsilon > 0)
\end{equation}
which are interesting in their own right. 
%Note that with condition (A 3) we may replace in \eqref{expected error II} and \eqref{deviations II} the parameter space $\Theta$ with $\overline{\Theta}$ so that we are not faced with difficulties of measurability. 
Note that (A 2) outrules trivial cases. 
\par
Convenient ways to find upper bounds of the expectations in \eqref{expected error II} may be provided by general devices from empirical process theory which are based on covering numbers for classes of Borel measurable mappings from $\RRR^{d}$ into $\RRR$ w.r.t. $L^{p}$-norms.  To recall these concepts adapted to our situation, let us fix any nonvoid set $\FFF$ of Borel measurable mappings from $\RRR^{d}$ into $\RRR$ and any probability measure $\MQ$ on $\cB(\RRR^{d})$ with metric $d_{\MQ,p}$ induced by the $L^{p}$-norm $\|\cdot\|_{\MQ,p}$ for $p\in [1,\infty[$.  
\begin{itemize}
\item \textit{Covering numbers for $\FFF$}\\
We use $N\big(\eta,\FFF,L^{p}(\MQ)\big)$ to denote the minimal number to cover $\FFF$ by closed $d_{\MQ,p}$-balls of radius $\eta > 0$ with centers in $\FFF$. We define $N\big(\eta,\FFF,L^{p}(\MQ)\big) := \infty$ if no finite cover is available. 
%\item \textit{Bracketing numbers for $\FFF$}\\
%Given two functions $\psi_{1},\psi_{2}:\RRR^{d}\rightarrow\RRR$ the \textit{bracket} $[\psi_{1},\psi_{2}]$ is the set of all mappings $\psi:\RRR^{d}\rightarrow\RRR$ satisfying $\psi_{1}(z)\leq\psi(z)\leq\psi_{2}(z)$ for every $z\in\RRR^{d}$. For $\varepsilon > 0$ it is an $\epsilon$\textit{-bracket in} $L^{2}(\MP^{Z})$ if $\psi_{1}, \psi_{2}$ are square $\MP^{Z}$-integrable with 
%$\|\psi_{2} - \psi_{1}\|_{\MP^{Z},2} < \varepsilon$. If  $\FFF$ may be covered by finitely many $\varepsilon$-brackets in $L^{2}(\MP^{Z})$ for every $\varepsilon > 0$, the corresponding
%minimal numbers will be denoted by $N_{[]}\big(\varepsilon,\FFF,L^{2}(\MP^{Z})\big)$. Otherwise $N_{[]}\big(\varepsilon,\FFF,L^{2}(\MP^{Z})\big) := \infty$.
\item An \textit{envelope} of $\FFF$ is defined to mean some Borel measurable mapping $C_{\FFF}:\RRR^{d}\rightarrow\RRR$ satisfying $\sup_{h\in\FFF}|h|\leq C_{\FFF}$. If an envelope $C_{\FFF}$ has strictly positive outcomes, we shall speak of a \textit{positive envelope}.
\item $\cM_{\textrm{\tiny fin}}$ denotes the set of all probability measures on $\cB(\RRR^{d})$ with finite support.
\end{itemize}
%Let $\cM_{\textrm{\tiny fin}}$ denote the set of all probability measures on $\cB(\RRR^{d})$ with finite support. The  

\medskip

For abbreviation let us introduce for a class $\FFF$ of Borel measurable functions from $\RRR^{d}$ into $\RRR$ with arbitrary positive envelope $C_{\FFF}$ of $\FFF$ the following notation
\begin{align}
&
\label{Entropie-Integral I}
J(\FFF,C_{\FFF},\delta) := \int_{0}^{\delta}\sup_{\MQ\in \cM_{\textrm{\tiny fin}}}\sqrt{\ln\big(2 N\big(\varepsilon~\|C_{\FFF}\|_{\MQ,2},\FFF,L^{2}(\MQ)\big)\big)}~d\varepsilon.
%\\ 
%&
%\label{Entropie-Integral II}
%J_{[]}(\FFF,C_{\FFF},\delta) := \int_{0}^{\delta}\sqrt{\ln\big(2 N_{[]}\big(\varepsilon~\|C_{\FFF}\|_{\MP^{Z},2},\FFF,L^{2}(\MP^{Z})\big)\big)}~d\varepsilon.
\end{align}
If the positive envelope $C_{\FFF}$ is $\MP^{Z}$-square integrable, then it is known that for every at most countable subset $\overline{\FFF}\subseteq\FFF$ the following inequality holds
\begin{align}
\EEE\Big[\sup_{h\in\overline{\FFF}}\Big|\frac{1}{n}\sum_{j=1}^{n}h(Z_{j}) - \EEE[h(Z_{1})]\Big|\Big]
&
\leq \nonumber \frac{\|C_{\FFF}\|_{\MP^{Z},2}}{\sqrt{n}}
8 \sqrt{2} J(\overline{\FFF},C_{\FFF},1)
\\
&
\leq \label{estimation with integrals}
\frac{16 \sqrt{2}~\|C_{\FFF}\|_{\MP^{Z},2}}{\sqrt{n}}~ J(\FFF,C_{\FFF},1/2)
\end{align} 
(see \cite[Remark 3.5.5]{GineNickl2016}).
\medskip

For our purposes the class $\FFF^{\Theta} := \{G(\theta,\cdot)\mid \theta\in\Theta\}$ is the relevant one. Then property (A 2) means nothing else but requiring a $\MP^{Z}$-integrable positive envelope of $\FFF^{\Theta}$. By \eqref{estimation with integrals} we may conclude immediately the following upper bounds 
for expectations in \eqref{expected error} and \eqref{expected error II}.
\begin{theorem}
\label{erstes Hauptresultat}
Let (A 1) - (A 3) be fulfilled, and let the envelope $\xi$ from (A 2) be square $\MP^{Z}$-integrable. Then with $\overline{\Theta}\subseteq\Theta$ from (A 3)
\begin{align*}
&\EEE\Big[\Big|\inf\limits_{\theta\in\Theta}~\frac{1}{n}~\sum_{j=1}^{n}G(\theta,Z_{j}) - \inf\limits_{\theta\in\Theta}\EEE\big[G(\theta,Z_{1})\big]\Big|\Big]\\
&\leq
\EEE\Big[\sup\limits_{\theta\in\overline{\Theta}}\Big|~\frac{1}{n}~\sum_{j=1}^{n}G(\theta,Z_{j}) - \EEE\big[G(\theta,Z_{1})\big]\Big|\Big]
%\\
%&
\leq 
\frac{16 \sqrt{2}~\|\xi\|_{\MP^{Z},2}}{\sqrt{n}}~ J(\FFF^{\Theta},\xi,1/2)~\mbox{for}~n\in\NNN.
\end{align*}
\end{theorem}
Let us turn over to bounds for \eqref{deviations} and \eqref{deviations II}. Since Talagrand introduced in \cite{Talagrand1994} and \cite{Talagrand1996} the first time his now famous concentration inequalities for empirical processes it is now well understood how to derive exponential estimates for the probabilities \eqref{deviations II}. They are essentially based on the expectations in \eqref{expected error II}. So henceforth we restrict considerations to ``small'' classes $\FFF_{\Theta}$ in the sense that $J(\FFF^{\Theta},\xi,1/2)$ is finite for some positive $\MP^{Z}$-square integrable envelope $\xi$ of $\FFF_{\Theta}$. We obtain the following result, using notation
\begin{equation}
\label{Restmenge}
B_{n}^{\xi} := \Big\{\frac{1}{n}\sum_{j=1}^{n}\xi(Z_{j})^{2}\leq  2\EEE[\xi(Z_{1})^{2}]\Big\}\quad(n\in\NNN)
\end{equation}
for any square $\MP^{Z}$-integrable strictly positive mapping 
$\xi: \RRR^{d}\rightarrow\RRR$, and
\begin{equation}
\label{allgemeiner Faktor}
\frak{f}_{n}:~ ]0,\infty[\rightarrow\RRR,~t\mapsto \frac{3~\sqrt{n}~\ln\big(1 + t/(5 t + 17)\big)}{2~(\sqrt{2n} + 1)}~\vee\frac{t}{5 t + 28}\quad\mbox{for}~n\in\NNN.
\end{equation} 
 %of $\FFF^{\Theta}$.
\begin{theorem}
\label{exponential bound}
Let (A 1) - (A 3) be satisfied, where the envelope $\xi$ from (A 2) is assumed to be square $\MP^{Z}$-integrable. Furthermore, let $\varepsilon > 0$ be fixed.  
Using notation \eqref{Entropie-Integral I}, if $J\big(\FFF^{\Theta},\xi,1/2\big)$ is finite, then with $\overline{\Theta}\subseteq\Theta$ from (A 3)
\begin{align*}
&\MP\Big(\Big\{\Big|\inf\limits_{\theta\in\Theta}~\frac{1}{n}~\sum_{j=1}^{n}G(\theta,Z_{j}) - \inf\limits_{\theta\in\Theta}\EEE\big[G(\theta,Z_{1})\big]\Big|\geq\varepsilon\Big\}\Big)\\
&\leq
\MP\Big(\Big\{\sup\limits_{\theta\in\overline{\Theta}}\Big|~\frac{1}{n}~\sum_{j=1}^{n}G(\theta,Z_{j}) - \EEE\big[G(\theta,Z_{1})\big]\Big|\geq\varepsilon\Big\}\Big)
\\
&
\leq 
\exp\left(\frac{-t~\sqrt{n}\varepsilon}{8 (t + 1)\|\xi\|_{\MP^{Z},2}}\cdot\frak{f}_{n}(t)\right)
+ \MP\big(\Omega\setminus B_{n}^{\xi}\big)
\end{align*}
holds for $t > 0$ and arbitrary $n\in\NNN$ with $\varepsilon > \eta_{t,n}$ as well as 
$n\geq \|\xi\|_{\MP^{Z},2}^{2}/2$, where 
$$
\eta_{t,n} := \|\xi\|_{\MP^{Z},2}/\sqrt{n} + 32\sqrt{2} (1+t) \|\xi\|_{\MP^{Z},2} J(\FFF^{\Theta},\xi,1/4)/\sqrt{n}.
$$
\end{theorem}
The proof of Theorem \ref{exponential bound} is an application of Talagrand's concentration inequalities along with the estimation \eqref{estimation with integrals}. The details are worked out in the Subsection \ref{zweite Schranke}. 
\begin{remark}
\label{Simplifications}
Let us point out some simplifications of Theorem \ref{exponential bound}.
\begin{itemize}
\item [1)] If the function $G$ is uniformly bounded by some positive constant $L$, then we may choose $\xi\equiv L$. Then $\eta_{t,n} = L[1 + 32\sqrt{2} (1+t)  J(\FFF^{\Theta},\xi,1/4)]/\sqrt{n}$ and $\Omega\setminus B_{n}^{\xi} = \emptyset$ for $t > 0$ and every $n\in\NNN$.
\item [2)] If $\xi_{1}(Z_{1})$ is integrable of order $4$, we may apply Chebychev's inequality to conclude
$$
\MP\big(\Omega\setminus B_{n}^{\xi}\big) \leq {\vari[\xi(Z_{1})^{2}]\over  n~\EEE[\xi(Z_{1})^{2}]^{2}}\quad\mbox{for}~n\in\NNN.
$$ 
\item [3)]
The upper estimate of the probability $\MP\big(\Omega\setminus B_{n}\big)$ in Theorem \ref{exponential bound} may be further improved if the random variable $\exp\big(\lambda\cdot\xi^{2}\big)$ is $\MP^{Z}$-integrable for some $\lambda > 0$. In this case In this case $\xi^{2}$ as well as $\exp\big(\lambda\cdot|\xi^{2}- \EEE[\xi(Z_{1})^{2}]|\big)$ are $\MP^{Z}$-integrable too, and 
\begin{align*}
M(\xi^{2}) 
&:= 
\sup_{k\in\NNN \atop k\geq 2}\Big(\Big|\EEE\big[\big(\xi(Z_{1})^{2} - \EEE[\xi(Z_{1})^{2}]\big)^{k}\big]\Big|/k!\Big)^{1/k}\\ 
&\leq 
\sqrt{\EEE\big[\exp\big(\lambda~|\xi(Z_{1})^{2} - \EEE[\xi(Z_{1})^{2}]|\big)\big]}/\lambda 
< \infty.
\end{align*}
Then 
the inequality $\EEE\big[t~\big(\xi(Z_{1})^{2} - \EEE[\xi(Z_{1})^{2}]\big)\big]\leq \exp(2 \delta^{2} t^{2})$ holds for any $\delta\geq M(\xi^{2})$ and every $-1/(2\delta)\leq t\leq 1/(2\delta)$ (see \cite[Theorem 1.3.2]{BuldyginKozachenko2000}). Hence we may draw on Theorem 2.6 from \cite{Petrov1995} to conclude
$$
\MP\big(\Omega\setminus B_{n}^{\xi}\big) \leq  \exp\left(- n \EEE[\xi(Z_{1})^{2}]^{2}/(8 \delta^{2})\right)\vee \exp\left(- n~\EEE[\xi(Z_{1})^{2}]/(4\delta)\right)
$$
for $n\in\NNN$, and any $\delta\geq M(\xi^{2})$.
\end{itemize}
\end{remark}
\begin{remark}
In \cite[Proposition 1]{GuiguesEtAl2017} the authors derive upper estimations
\begin{align*}
&
\MP\Big(\Big\{\inf\limits_{\theta\in\Theta}~\frac{1}{n}~\sum_{j=1}^{n}G(\theta,Z_{j}) - \inf\limits_{\theta\in\Theta}\EEE\big[G(\theta,Z_{1})\big]\leq -\varepsilon/\sqrt{n}\Big\}\Big)\leq \exp(- \varepsilon^{2}~M_{l})\\
&
\MP\Big(\Big\{\inf\limits_{\theta\in\Theta}~\frac{1}{n}~\sum_{j=1}^{n}G(\theta,Z_{j}) - \inf\limits_{\theta\in\Theta}\EEE\big[G(\theta,Z_{1})\big]\geq\varepsilon/\sqrt{n}\Big\}\Big)\leq \exp(- \varepsilon^{2}~M_{u})
\end{align*}
with certain constants $M_{l}, M_{u}$ for small  positive $\varepsilon$. They assume $G$ to be convex in $\Theta$ such that the goal function of \eqref{optimization risk neutral} is differentiable, and for some $\lambda_{1} > 0$ 
$$
\sup_{\theta\in\Theta}\EEE\big[\exp\big(\lambda^{2}_{1}~\{G(\theta,\cdot) - \EEE[G(\theta,\cdot)]\}^{2}\big)\big]\leq \exp(1).
$$
In addition, with some positive number $\lambda_{2}$, they also impose an analogous condition on the  random maximal distances between the derivatives of the goal function of \eqref{optimization risk neutral} and the subgradients of the goal functions in \eqref{SAAriskneutral}. The constant $M_{l}$ in the upper estimates is dependent on $\lambda_{1}$, whereas $M_{u}$ relies on $\lambda_{1}, \lambda_{2}$ and some further auxiliary constants.
\smallskip

Theorem \ref{exponential bound} and the corresponding result in \cite{GuiguesEtAl2017} allow both to find for small deviation $\varepsilon > 0$ upper estimates of the sample size $n$ ensuring that the probability of deviations \eqref{deviations} does not exceed given levels of tolerance. The special feature of the result in \cite{GuiguesEtAl2017} is that the upper estimate of the sample size may be chosen independently of the dimension of the parameters (see \cite[Discussion 2.1.3, (3)]{GuiguesEtAl2017}). In contrast, the estimates which may be derived from Theorem \ref{exponential bound} rely on the integrals $J(\FFF^{\Theta},\xi,1/4)$ which often increase with the dimension (see e.g. Propositions \ref{Hoelder-Bedingung auxiliary} , \ref{startingpoint} below).

\end{remark}

\medskip

There exists some interesting link between the probability deviations \eqref{deviations II} and the tail functions for the absolute error of the solutions of the SAA \eqref{SAAriskneutral}. In \cite{Pflug1999} this was pointed out the first time, and the relationship was further systemized in \cite{Roemisch2003}. To go into more detail let us consider any $m$-dimensional random vector $\widehat{\theta}_{n}$ which is a (random) solution of the SAA-optimization problem \eqref{SAAriskneutral} with sample size $n$. Furthermore, let $\psi_{\Theta}: \Theta\rightarrow\RRR$ be defined by $\psi_{\Theta}(\theta) = \EEE[G(\theta,Z_{1})]$. 
\begin{itemize}
\item If $\psi_{\Theta}$ is lower semicontinuous with a unique minimal point $\theta^{*}$,  and if it satisfies some specific growth condition, then under (A 3) there is some constant $L$ such that 
$$
\MP\big(\big\{\|\widehat{\theta}_{n} - \theta^{*}\| > \varepsilon\big\}\big)\leq 
\MP\Big(\Big\{\sup\limits_{\theta\in\overline{\Theta}}\Big|~\frac{1}{n}~\sum_{j=1}^{n}G(\theta,Z_{j}) - \EEE\big[G(\theta,Z_{1})\big]\Big| > L ~\varepsilon^{2}/2\Big\}\Big)
$$
for any $\varepsilon > 0$ (see \cite[p. 66]{Pflug1999}).
\item If $G$ is random lower semicontinuous, if $\Psi_{\Theta}$ is Lipschitz continuous, and if (A 3) is satisfied, then the set $S(\psi_{\Theta})$ of minimizers of $\psi_{\Theta}$ is nonvoid, and there is some unbounded strictly increasing mapping $\varphi_{\Theta}: [0,\infty[\rightarrow [0,\infty[$ with $\varphi_{\Theta}(0) = 0$ such that
\begin{align*}
&
\MP\big(\big\{\inf_{\theta\in S(\psi_{\Theta})}\|\widehat{\theta}_{n} - \theta\| > \varepsilon/\sqrt{n}\big\}\big)\\
&\leq 
\MP\Big(\Big\{\sup\limits_{\theta\in\overline{\Theta}}\Big|~\frac{1}{n}~\sum_{j=1}^{n}G(\theta,Z_{j}) - \EEE\big[G(\theta,Z_{1})\big]\Big| > \varphi_{\Theta}(\varepsilon/\sqrt{n})\Big\}\Big)\quad\mbox{for}~\varepsilon > 0
\end{align*}
(see \cite[Theorem 50]{Roemisch2003}).
\end{itemize}  
In view of these results the second inequality in Theorem \ref{exponential bound} might be utilized to derive upper estimates for the tail function of the absolute error of 
$\widehat{\theta}_{n}$ dependent on the sample size $n$. This was already recognized in the contributions \cite{Pflug1999} and \cite{Roemisch2003}, referring to concentration inequalities in \cite{Talagrand1994}, where however unspecified univeral constants are used. 
\bigskip

As an easy consequence of Theorem \ref{exponential bound} we may provide the following simple criterion to ensure uniform tightness of the sequence 
$$
\Big(\sqrt{n}\Big[\inf_{\theta\in\Theta}~\frac{1}{n}~\sum_{j=1}^{n}G(\theta,Z_{j}) - \inf_{\theta\in\Theta}\EEE\big[G(\theta,Z_{1})\big]\Big]\Big)_{n\in\NNN}.
$$
The new point is that we do not require the paths $G(\cdot,z)$ to satisfy certain analytical properties in advance, as often in the literature on the SAA method (e.g. in \cite{ShapiroEtAl} or \cite{EichhornRoemisch2007}).
\begin{theorem}
\label{tightness}
Let (A 1) - (A 3) be fulfilled with $\xi$ from (A 2) being square $\MP^{Z}$-integrable. 
Using notation \eqref{Entropie-Integral I}, if $J\big(\FFF^{\Theta},\xi,1/2\big)$ is finite,
then the sequence 
$$
\Big(\sqrt{n}\Big[\inf_{\theta\in\Theta}~\frac{1}{n}~\sum_{j=1}^{n}G(\theta,Z_{j}) - \inf_{\theta\in\Theta}\EEE\big[G(\theta,Z_{1})\big]\Big]\Big)_{n\in\NNN}.
$$
is uniformly tight.
\end{theorem}
\begin{proof}
%Set $B_{n}^{\xi} := \big\{\frac{1}{n}\sum_{j=1}^{n}\xi(Z_{j})\leq 2 \EEE[\xi(Z_{1})^{2}]\big\}$ for $n\in\NNN$. 
Fix any $n\in\NNN$ with $n\geq \|\xi\|_{\MP^{Z},2}^{2}/2$. Then with $B_{n}^{\xi}$ as defined in \eqref{Restmenge} the application of Theorem \ref{exponential bound} yields
\begin{align*}
&
\MP\Big(\Big\{\sqrt{n}\Big|\inf\limits_{\theta\in\Theta}~\frac{1}{n}~\sum_{j=1}^{n}G(\theta,Z_{j}) - \inf\limits_{\theta\in\Theta}\EEE\big[G(\theta,Z_{1})\big]\Big|\geq\varepsilon\Big\}\Big)\\
&
\leq 
\exp\left(\frac{-t^{2}~\varepsilon}{8 (t + 1) (5 t + 28) \|\xi\|_{\MP^{Z},2}}\right)
+ \MP\big(\Omega\setminus B_{n}^{\xi}\big)
\end{align*}
for $t > 0$ and every $\varepsilon > \|\xi\|_{\MP^{Z},2} + 32\sqrt{2} (1+t) \|\xi\|_{\MP^{Z},2} J(\FFF^{\Theta},\xi,1/4)$. Furthermore we have convergence $\MP\big(\Omega\setminus B_{n}^{\xi}\big)\to 0$ by the law of large numbers. Thus
$$
\lim_{\varepsilon\to\infty}\limsup_{n\to\infty}\MP\Big(\Big\{\sqrt{n}\Big|\inf\limits_{\theta\in\Theta}~\frac{1}{n}~\sum_{j=1}^{n}G(\theta,Z_{j}) - \inf\limits_{\theta\in\Theta}\EEE\big[G(\theta,Z_{1})\big]\Big|\geq\varepsilon\Big\}\Big) = 0
$$
which completes the proof.
\end{proof}
All the results within this section crucially require $J(\FFF^{\Theta},\xi,1/2)$ to be finite. This property is always satisfied if the involved covering numbers have polynomial rates. Indeed this relies on the observation, that by using change of variable formula several times along with integration by parts, we obtain
\begin{equation}
\label{Integralabschaetzung}
\int_{0}^{1}\sqrt{v\ln(K/\varepsilon)}~d\varepsilon\leq 2\sqrt{v \ln(K)}\quad\mbox{for}~v\geq 1, K\geq e.
\end{equation}
Inequality \eqref{Integralabschaetzung} may be applied %to simplify \eqref{endliche Entropie-Integrale} 
if there exist $K\geq e, v\geq 1$ such that the following condition is satisfied
$$
N\big(\varepsilon~\|C_{\FFF^{\Theta}}\|_{\MQ,2},\FFF^{\Theta},L^{2}(\MQ)\big)\big)\leq (K/\varepsilon)^{v}\quad\mbox{for}~\MQ\in\cM_{\textrm{\tiny fin}}\quad\mbox{and}~\varepsilon\in ]0,1[.
$$
In the rest of this section we shall utilize \eqref{Integralabschaetzung} to give explicit upper estimates of the terms $J(\FFF^{\Theta}, C_{\FFF^{\Theta}},\delta)$ if the objective $G$ satisfies specific analytical properties. 
\medskip

Denoting the Euclidean metric on $\RRR^{m}$ by $d_{m,2}$, we start with the following condition
\begin{enumerate}
\item[(H)] There exist some $\beta\in ]0,1]$ and a square $\MP^{Z}$-integrable strictly positive mappings $C:\RRR^{d}\rightarrow ]0,\infty[$ such that 
$$
\big|G(\theta,z) - G(\vartheta,z)\big|\leq C(z)~d_{m,2}(\theta,\vartheta)^{\beta}\quad\mbox{for}~z\in\RRR^{d}, \theta, \vartheta\in\Theta.
$$
\end{enumerate}
Under (H) explicit upper estimates for the terms $J(\FFF^{\Theta},\xi,\delta)$ are provided by the following result.
\begin{proposition}
\label{Hoelder-Bedingung auxiliary}
Let condition (H) be fulfilled with $\beta\in ]0,1]$ and square $\MP^{Z}$-integrable strictly positive mapping $C$. Furthermore, let $G(\theta,\cdot)$ be Borel measurable for every $\theta \in\Theta$. In addition let $\Delta(\Theta)$ stand for the diameter of $\Theta$ w.r.t. the Euclidean metric $d_{m,2}$. Then requirement (A 3) is met. Moreover, in case of $\Delta(\Theta) > 0$, if $G(\overline{\theta},\cdot)$ is square $\MP^{Z}$-integrable for some $\overline{\theta}\in\Theta$, the mapping $\xi := C~\Delta(\Theta)^{\beta} + |G(\overline{\theta},\cdot)|$ is square $\MP^{Z}$-integrable, satisfying property (A 2) and 
$$
\sup_{\MQ\in \cM_{\textrm{\tiny fin}}}N\big(\varepsilon \|\xi\|_{\MQ,2},\FFF^{\Theta},L^{2}(\MQ)\big)\leq \big(8 + \varepsilon^{1/\beta}\big)^{m}/\varepsilon^{m/\beta}\quad\mbox{for}~\varepsilon > 0.
$$
In particular
$$
J(\FFF^{\Theta},\xi,\delta)\leq 2\delta\sqrt{(3m + 1)\ln(2) + \frac{m}{\beta}\ln(2/\delta)}\quad\mbox{for}~\delta \in ]0,1/2].
$$
\end{proposition}
For the proof see Subsection \ref{Beweis Hoelder Hilfsresultat}.
\begin{remark}
Proposition \ref{Hoelder-Bedingung auxiliary} tells us that under (H) the Theorems \ref{exponential bound}, \ref{tightness} carry over immediately, using the estimates from Proposition \ref{Hoelder-Bedingung auxiliary}.
\end{remark}
Next, let us consider objective $G$ having the following kind of structure of piecewise H\"older continuity. 
\begin{enumerate}
\item [(PH)] $
G(\theta,z) = \sum\limits_{i = 1}^{r}\left(\min_{l=1,\dots,s_{i}}
\eins_{I_{il}}\big(\Lambda_{i,l}(\theta, z) + a^{i}_{l}\big)\right)\cdot 
G^{i}(\theta, z),
$ 
where
\begin{itemize}
\item $r, s_{1},\dots,s_{r}\in\NNN$,
\item $G^{i}$ satisfies (A 1), and (H) with $\beta_{i}\in ]0,1]$ as well as strictly positive square $\MP^{Z}$-integrable $C_{i}:\RRR^{d}\rightarrow\RRR$ for 
$i\in\{1,\ldots,r\}$,
\item $\Lambda_{il}:\RRR^{m}\times\RRR^{d}\rightarrow\RRR$ Borel measurable with $\Lambda_{il}(\cdot,z)$ affine linear for $z\in\RRR^{d}$ ($i\in\{1,\ldots,r\}$, $l\in\{1,\ldots,s_i\}$),
\item $a^{i}_{l}\in\RRR$ for $i\in\{1,\dots,r\}, l\in\{1,\dots,s_{i}\}$,
\item $I_{il} = ]0,\infty[$ or $I_{il} = [0,\infty[$ for $i\in\{1,\dots,r\}$ and $l\in\{1,\dots,s_{i}\}$,
\item The set
$$
\Big\{\bigcap\limits_{l=1}^{s_{i}}\big\{\Lambda_{il}(\theta, \cdot) + a^{i}_{l}\in I_{il}\}\mid i\in\{1,\ldots,r\}, l\in\{1,\ldots,s_{i}\big\}\Big\}
$$
is a partition of $\RRR^{d}$.
%\item $\min\limits_{l=1,\dots,s_{i}}\eins_{I_{il}}\big(\Lambda_{il}(\theta, z) + a^{i}_{l}\big)\cdot \min\limits_{l=1,\dots,s_{j}}\eins_{I_{jl}}\big(\Lambda_{jl}(\theta, z) + a^{j}_{l}\big) = 0$ for $i\not= j$,
%\item $\sum\limits_{i=1}^{r}\min\limits_{l=1,\dots,s_{i}}\eins_{I_{il}}\big(\Lambda_{il}(\theta, z) + a^{i}_{l}\big) = 1$.
\end{itemize}
\end{enumerate}
In two stage mixed-integer programs the goal functions typically may be represented in this way if the random vector $Z$ has compact support (see \cite[p. 121]{EichhornRoemisch2007}). More precisely, under the conditions of relative complete recourse and dual feasibility, we may find $p, s\in\NNN$, affine linear mappings $h:\RRR^{d}\rightarrow\RRR^{p}$, $\Lambda:\RRR^{d}\rightarrow\RRR^{p\times m}$, linear mappings $T:\RRR^{m}\rightarrow\RRR$, $L^{i}_{l}: \RRR^{p}\rightarrow\RRR$, and Lipschitz-continuous mappings $\varphi_{i}:\RRR^{p}\rightarrow\RRR$ satisfying
\begin{itemize}
\item $s_{i} = s$ for $i=1,\ldots,r$,
\item $G^{i}(\theta,z) = T(\theta) + \varphi_{i}\big(h(z) - \Lambda(z)\cdot\theta\big)$ for $i\in\{1,\ldots,r\}$,
\item $\Lambda_{il}(\theta,z) = L^{i}_{l}\big(- h(z) + \Lambda(z)\cdot\theta\big)$ for $i\in\{1,\ldots,r\}$, $l\in\{1,\ldots,s_{i}\}$.
\end{itemize}
Based upon this representation of $G$ the authors in \cite{EichhornRoemisch2007} not only show uniform tightness of the sequence in Theorem \ref{tightness} but even derive asymptotic distributions. Their line of reasoning relies also  on covering numbers for the class $\FFF_{\Theta}$ and finiteness of the integrals in \eqref{Entropie-Integral I}.
\medskip

Note that $G$ satisfying condition (PH) does not have continuity in $\theta$ in advance.
\par
%For preparation to find explicit upper estimations of the terms $J(\FFF^{\Theta},\xi,\delta)$ we may observe by compactness of $\Theta$ along with the continuity of the mappings $\Lambda_{i}$
%$$
%\eta_{i}^{G} := \sup_{\theta\in\Theta}|\Lambda_{i}\circ T(\theta) + b_{i}| + \eins_{\{0\}}\Big(\sup_{\theta\in\Theta}|\Lambda_{i}\circ T(\theta) + b_{i}|\Big) < \infty\quad\mbox{for}~i\in\{1,\ldots,r\}.
%$$
For abbreviation we set 
%\begin{align*}
%&&
%\label{fi}
$f^{i}(\theta,z) := \min_{l=1,\dots,s_{i}}\eins_{I_{il}}\big(\Lambda_{il}(\theta, z) + a^{i}_{l}\big)$
%\quad\mbox{and}
%\quad G^{i}(\theta,z) := \Lambda_{i}\big(T(\theta) + z)\big) + b_{i}%\\
%&&
%\label{Gi}
%G^{i}(\theta,z) := \Lambda_{i}\big(T(\theta) + z)\big) + b_{i}\\
%\end{align*}
for $i\in\{1,\ldots,r\}$, and we introduce the associated function classes
$$
\FFF_{\textrm{\tiny PH}}^{i} := \big\{f^{i}(\theta,\cdot)\mid \theta\in\Theta\big\}\quad\mbox{and}\quad\overline{\FFF}_{\textrm{\tiny PH}}^{i} := \big\{G^{i}(\theta,\cdot)\mid \theta\in\Theta\big\}\quad i\in\{1,\ldots,r\}.
$$
Note that the classes $\FFF_{\textrm{\tiny PH}}^{i}$ are uniformly bounded by $1$. The following result gives an upper estimate of the terms $J(\FFF^{\Theta},\xi,\delta)$.
\begin{proposition}
\label{startingpoint} 
Let $\Delta(\Theta)$ denote the diameter of $\Theta$ w.r.t. the Euclidean metric. 
The mappings $f^{i}(\theta,\cdot)$ and $G^{i}(\theta,\cdot)$ are Borel measurable for $\theta\in\Theta$ and $i\in\{1,\ldots,r\}$. In particular assumption (A 1) holds. Moreover, in case of $\Delta(\Theta) > 0$, if $G^{1}(\overline{\theta},\cdot),\ldots,G^{r}(\overline{\theta},\cdot)$ are square $\MP^{Z}$-integrable for some $\overline{\theta}\in\Theta$, and if $\xi_{1},\ldots,\xi_{r}$ denote bounded positive envelopes of $\FFF_{\textrm{\tiny PH}}^{1},\ldots,\FFF_{\textrm{\tiny PH}}^{r}$ respectively, then $\xi := \sum_{i=1}^{r}\xi_{i}\cdot\big(\Delta(\Theta)^{\beta_{i}}~C_{i}(\cdot) + |G^{i}(\overline{\theta},\cdot)|\big)$ is square $\MP^{Z}$-integrable satisfying (A 2) and 
\begin{align*}
&
J(\FFF^{\Theta},\xi,\delta)
\\
&
\leq 2\delta\sqrt{r + 2 r~m~\ln(3) +  m \ln(4 r/\delta) \sum_{i=1}^{r} 1/\beta_{i} + \ln(2) + [5 + 2 \ln(4 r/\delta)]~ (m + 2) \sum_{i=1}^{r}s_{i}}
\end{align*}
for $\delta\in ]0,1]$.
\end{proposition}
The involved proof is delegated to Subsection \ref{proof of Proposition starting point}.
\begin{remark}
\label{Remark zu PL}
In view of Proposition \ref{startingpoint} the only critical condition left is (A 3) in order to apply our main results. If $G^{1}(\overline{\theta},\cdot),\ldots,G^{r}(\overline{\theta},\cdot)$ are $\MP^{Z}$-integrable for some $\overline{\theta}\in\Theta$, it is a routine excercise to show that (A 3) may be guaranteed e.g. for any at most countable dense subset $\overline{\Theta}\subseteq\Theta$ by the following property.
\smallskip

\noindent
\begin{itemize}
\item [(*)] $\left\{z\in\RRR^{d}\mid \Lambda_{il}(\theta,z) = - a^{i}_{l}~\mbox{for some}~ \theta\in\Theta\right\}$ is contained in a $\MP^{Z}$-null set for $i = 1,\ldots,r$ and $l\in\{1,\ldots,s_{i}\}$ with $I_{il} = [0,\infty[$. 
\end{itemize}
\smallskip

\noindent
For the application of the main results we may invoke the estimates from Proposition \ref{startingpoint}.
\end{remark}
%%%%%%%%%%%%%%%%%%%%%%%%%%%%%%%%%%%%%%%%%%%%%%%%%%%%%%%%%%%%%%%%%%%%%%%%%%%%%%%%%%%%%%%%%%%%%%%%%%%%%%%%%%%%%%%%%%%%%%%%%%%%%%%%%%%%%%%%%%%%%%%%%%%%%%%%%%%%%%%%%%%%%%%%%%%%%%%%%%%%%%%%%%%%%%%%%%%%%%%%%%%%%%%%%%%%%%%%%%%%%%%%%%%%%%%%%%%%%%%%%%%%%%%%%%%%%%%%%%%%%%%%%%%%%%%%%
\section{Error estimates under mean upper semideviations}
\label{upper semideviations}
Let $L^p(\Omega,\cF,\pr)$ denote the usual $L^{p}$-space on $\OFP$ ($p\in [0,\infty[$), where we tacitely identify random variables which are different on $\MP$-null sets only. The space $L^{p}\OFP$ is endowed with the usual $L^{p}$-norm $\|\cdot\|_{p}$.
\par
We want to study the risk averse stochastic program (\ref{optimization risk averse}), where  
in the objective the functional $\rho$ is a \textit{mean upper semideviation} also known as \textit{upper semideviation-corrected expectation}. This means that for $p\in [1,\infty[$ and $a\in ]0,1]$ the functional $\rho = \rho_{p,a}$ is defined as follows
$$
\rho_{p,a}:L^{p}\OFP\rightarrow\RRR,~X\mapsto\EEE[X] + a \|\big(X - \EEE[X]\big)^{+}\|_{p}.
$$
The number $p$ is the order of the mean upper semideviation and $a$ is the weight of the correction from the expecation. It is well-known that mean upper semideviations are increasing w.r.t. the increasing convex order (cf. e.g. \cite[Theorem 6.51 along with Example 6.23 an Proposition 6.8]{ShapiroEtAl}). They are also distribution-invariant so that we may define the associated functional $\cR_{p,a}$ on the set of distributions functions of random variables with absolute moments of order $p$. So the subject of this section is the optimization problem 
\begin{equation*}
\inf_{\theta\in\Theta}\cR_{p,a}\big(F_{\theta}\big),
\end{equation*}
where $F_{\theta}$ stands for the distribution function of $G(\theta,Z)$ for 
$\theta\in\Theta$. Introducing the notation 
\begin{equation}
\label{Hilfsgoals}
G_{p}:\Theta\times\RRR^{d}\rightarrow\RRR,~(\theta,z)\mapsto \left[\big(G(\theta,z) - \EEE[G(\theta,Z_{1})]\big)^{+}\right]^{p}\quad (p\in [1,\infty[).
\end{equation}
we may describe this optimization also in the following way
\begin{equation}
\label{optimization general III}
\inf_{\theta\in\Theta}\cR_{\rho_{p,a}}\big(F_{\theta}\big) = 
\inf_{\theta\in\Theta}\left\{\EEE[G(\theta,Z_{1})] + a\Big(\EEE[G_{p}(\theta,Z_{1})]\Big)^{1/p}\right\}.
\end{equation}
Then the stochastic objective of the approximative problem according to the SAA method has the following representation.
\begin{equation}
\label{SAA upper semideviations}
%\inf_{\theta\in\Theta}
\cR_{\rho_{p,a}}\big(\hat{F}_{n,\theta}\big) = 
%\inf_{\theta\in\Theta}
\Big\{\frac{1}{n}\sum_{j=1}^{n}G(\theta,Z_{j}) + a\Big(\frac{1}{n}\sum_{j=1}^{n}\big(\big[G(\theta,Z_{j}) - \frac{1}{n}\sum_{i=1}^{n}G(\theta,Z_{i})\big]^{+}\big)^{p}\Big)^{1/p}\Big\}
%\inf_{\theta\in\Theta}\left(\frac{1}{n}\sum_{j=1}^{n}G(\theta,Z_{j}) + \Big[\Big(\frac{1}{n}\sum_{j=1}^{n}\big[G(\theta,Z_{j}) - \frac{1}{n}\sum_{i=1}^{n}G(\theta,Z_{i})\big]^{+}\Big)^{p}\Big]^{1/p}\right)
\end{equation}
We shall look at bounds for the deviation probabilities \eqref{Kreuz} w.r.t. $\cR_{\rho_{p,a}}$.
%$$
%\MP\Big(\Big\{\big|\inf_{\theta\in\Theta}\cR_{\rho_{p,a}}\big(\hat{F}_{n,\theta}\big) - \inf_{\theta\in\Theta}\cR_{\rho_{p,a}}\big(F_{\theta}\big)\big|\geq\varepsilon\Big\}\Big)\quad(n\in\NNN,\varepsilon > 0).
%$$
It is intended to utilize results for risk neutral case presented in Section \ref{error risk neutral}. The key is the following observation based on the notation \eqref{Hilfsgoals}.
\begin{lemma}
\label{Rueckfuehrung}
Let (A 1) be fulfilled, and let $\xi$ be an envelope of $\FFF^{\Theta}$ which is $\MP^{Z}$-integrable of order $p\in [1,\infty[$. Then the optimal values of the problems \eqref{optimization general III} and \eqref{SAA upper semideviations} are finite. Moreover, 
for any nonvoid subset $\overline{\Theta}\subseteq\Theta$ and arbitrary $n\in\NNN$, 
$\varepsilon > 0$ as well as  $a\in ]0,1]$
\begin{align*}
\Big\{\big|\inf_{\theta\in\overline{\Theta}}\cR_{\rho_{p,a}}\big(\hat{F}_{n,\theta}\big) - \inf_{\theta\in\overline{\Theta}}\cR_{\rho_{p,a}}\big(F_{\theta}\big)\big|\geq\varepsilon\Big\}
\subseteq D_{n,\varepsilon,a}^{\overline{\Theta}}\cup \overline{D}_{n,\varepsilon,p,a}^{\overline{\Theta}},
\end{align*}
holds, where
\begin{align*}
&
D_{n,\varepsilon,a}^{\overline{\Theta}} := \Big\{\sup_{\theta\in\overline{\Theta}}\big|\frac{1}{n}\sum_{j=1}^{n}G(\theta,Z_{j}) - \EEE[G(\theta,Z_{1})]\big|\geq\varepsilon/(2+2a)\Big\},
\\
&
\overline{D}_{n,\varepsilon,p,a}^{\overline{\Theta}} := \Big\{\sup_{\theta\in\overline{\Theta}}\big|\frac{1}{n}\sum_{j=1}^{n}G_{p}(\theta,Z_{j}) - \EEE[G_{p}(\theta,Z_{1})]\big|\geq\big(\varepsilon/[2 a]\big)^{p}\Big\}.
\end{align*}
\end{lemma}
The proof may be found in Subsection \ref{Beweise upper semideviations}. 
\medskip

In the next step we want to reduce simultaneously the optimization problems \eqref{optimization general III} and \eqref{SAA upper semideviations} to at most countable parameter subsets of $\Theta$. This will be achieved by the following assumption which strengthens (A 3). 
\begin{itemize}
\item [(A 3')] There exist some at most countable subset $\overline{\Theta}\subseteq\Theta$ and $(\MP^{Z})^{n}$-null sets $N_{n}$ $(n\in\NNN)$ such that for $z_{1},\ldots,z_{n}\in\RRR^{dn}\setminus N_{n}$ and $\theta\in\Theta$
$$
\inf_{\vartheta\in\overline{\Theta}}\Big\{\EEE[|G(\vartheta,Z_{1}) - G(\theta,Z_{1})|] +  \max_{j\in\{1,\ldots,n\}}\big|G(\theta,z_{j}) - G(\vartheta,z_{j})\big|\Big\} = 0.
$$
\end{itemize}
\begin{remark}
\label{Bemerkung I}
Under (A 2), property (A 3') may be checked easily if condition (H) is satisfied. If $G$ has representation (PL), and if the involved linear mappings $\Lambda_{1},\ldots,\Lambda_{r}$ are $\MP^{Z}$-integrable, then (A 3') holds under (*) from Remark \ref{Remark zu PL}.
\end{remark}
\begin{lemma}
\label{ReduktionCountable}
Let (A 1) and (A 3') be satisfied, and let $\xi$ be some positive envelope of $\FFF^{\Theta}$ which is $\MP^{Z}$-integrable of order $p\in [1,\infty[$. Then with the at most countable subset $\overline{\Theta}\subseteq\Theta$ and the $(\MP^{Z})^{n}$-null sets $N_{n}$ from (A 3') the following statements hold.
\begin{itemize}
\item [1)] $\inf\limits_{\theta\in\Theta}\cR_{\rho_{p,a}}\big(F_{\theta}\big) = \inf\limits_{\theta\in\overline{\Theta}}\cR_{\rho_{p,a}}\big(F_{\theta}\big)$ for $a\in ]0,1]$.
\item [2)] For $n\in\NNN$, $\theta\in\Theta$ and $(z_{1},\ldots,z_{n})\in\RRR^{dn}\setminus N_{n}$
\begin{align*}
&
\inf_{\vartheta\in\overline{\Theta}}\big|\EEE[G(\vartheta,Z_{1})] - \EEE[G(\vartheta,Z_{1})]\big| = \inf_{\vartheta\in\overline{\Theta}}\max_{j=1,\ldots,n}\big|G(\vartheta,z_{j}) - G(\vartheta,z_{j})\big| = 0,\\
&
\inf_{\vartheta\in\overline{\Theta}}\big|\EEE[G_{p}(\vartheta,Z_{1})] - \EEE[G_{p}(\vartheta,Z_{1})]\big| = \inf_{\vartheta\in\overline{\Theta}}\max_{j=1,\ldots,n}\big|G_{p}(\vartheta,z_{j}) - G_{p}(\vartheta,z_{j})\big| = 0. 
\end{align*}
\item [3)] If $n\in\NNN$, and if $a\in ]0,1]$, then 
$
\inf\limits_{\theta\in\Theta}\cR_{\rho_{p,a}}\big(\hat{F}_{n,\theta}\big)  = \inf\limits_{\theta\in\overline{\Theta}}\cR_{\rho_{p,a}}\big(\hat{F}_{n,\theta}\big)~\MP-\mbox{a.s.}.
$
\end{itemize}
\end{lemma}
The proof is provided in Subsection \ref{Beweise upper semideviations}.
\medskip

Lemma \ref{Rueckfuehrung} suggests to apply the results from Section \ref{error risk neutral} simultaneously to the function classes $\FFF^{\Theta}$ and 
$\FFF^{\Theta,p}:=\{G_{p}(\theta,\cdot)\mid\theta\in\Theta\}$ ($p\in [1,\infty[$). However, we want to describe the involved terms $J(\FFF^{\Theta,p},C_{\FFF^{\Theta,p}},\delta)$ by means of the terms $J(\FFF^{\Theta},C_{\FFF^{\Theta}},\delta)$ associated with the genuine objective $G$.
This will be done in the following auxiliary result. 
%There we shall denote the smallest integer which is smaller or equal to some $p\in [1,\infty[$ by $\lceil p\rceil$.
\begin{lemma}
\label{auxiliary I}
Let (A 1) be fulfilled, and let $\xi$ be a positive envelope of $\FFF^{\Theta}$ which is $\MP^{Z}$-integrable of order $2 (p + 1)$ for some $p\in [1,\infty[$. Then $\xi_{p} := \big[\xi + \big(\EEE[\xi(Z_{1})]\vee 1\big)\big]^{p + 1 }$ is a square $\MP^{Z}$-integrable  positive envelope 
of $\FFF^{\Theta,p}$ satisfying
$$
J(\FFF^{\Theta,p},\xi_{p},\delta)\leq \sqrt{2}~2^{p+2}J(\FFF^{\Theta},\xi,\delta/2^{p+2}) + \sqrt{2}~\delta~ [\sqrt{\ln(2)} + 2\sqrt{\ln\big(2^{p+4}/\delta\big)}]
$$
for $~\delta\in ]0,1[$.
\end{lemma}
The proof is delegated to Subsection \ref{Beweise upper semideviations}.
\medskip

Now, we are prepared to formulate and prove the main result on error estimates under upper semideviations. 
\begin{theorem}
\label{errorRatesUpperSemideviations}
Let (A 1), (A 2), (A 3') be fulfilled, where the Borel measurable mapping $\xi$ from (A 2) is integrable of order $2 (p + 1)$ for some $p\in [1,\infty[$. Setting $\xi_{p} := \linebreak\big[\xi + \big(\EEE[\xi(Z_{1})]\vee 1\big)\big]^{p + 1}$, and assuming $J(\FFF^{\Theta},\xi,1/4) < \infty$ the following statements are valid.
\begin{itemize}
\item [1)] For $\varepsilon, t > 0$, $n\in\NNN$ with $n\geq\max\big\{ \|\xi_{p}\|_{\MP^{Z},2}^{2}/2, [1 + 32 \sqrt{2} (t + 1) J(\FFF^{\Theta},\xi,1/4)]^{2}\big\}$, and $a\in ]0,1]$ 
the inequality 
\begin{align*}
& 
\MP\Big(\Big\{\big|\inf_{\theta\in\Theta}\cR_{\rho_{p,a}}\big(\hat{F}_{n,\theta}\big) - \inf_{\theta\in\Theta}\cR_{\rho_{p,a}}\big(F_{\theta}\big)\big|\geq\varepsilon\Big\}\Big)\\
&\leq 
\exp\left(\frac{-t~\sqrt{n}\varepsilon}{16 (a + 1) (t + 1) \|\xi\|_{\MP^{Z},2}}\cdot\frak{f}_{n}(t)\right)
+ 
\exp\left(\frac{-t~\sqrt{n}\varepsilon^{p}}{2^{p+3} a^{p} (t + 1) \|\xi_{p}\|_{\MP^{Z},2}}\cdot\frak{f}_{n}(t)\right)\\
&\quad
+ 
\MP\big(\Omega\setminus B_{n}^{\xi}\big) + \MP\big(\Omega\setminus B_{n}^{\xi_{p}}\big),
\end{align*}
holds if 
$$
\varepsilon > \frac{2 (1 + a) 32^{1/p} (t + 1)^{1/p} \|\xi_{p}\|_{\MP^{Z},2}^{1/p}}{n^{1/(2p)}}~\big[1 + \sqrt{p+6} + 2^{p+3} J(\FFF^{\Theta},\xi,1/2^{p+4})\big]^{1/p}.
$$ 
%as well as $n\geq\max\big\{ \|\xi_{p}\|_{\MP^{Z},2}^{2}/2, [1 + 32 \sqrt{2} J(\FFF^{\Theta},\xi,1/4)]^{2}\big\}$. 
Here $B_{n}^{\xi}$ and $B^{\xi_{p}}_{n}$ are defined according to \eqref{Restmenge}.
\item [2)] The sequence 
$$
\Big(\sqrt{n}~\big[\inf_{\theta\in\Theta}\cR_{\rho_{p,a}}\big(\hat{F}_{n,\theta}\big) - \inf_{\theta\in\Theta}\cR_{\rho_{p,a}}\big(F_{\theta}\big)\big]\Big)_{n\in\NNN}
$$
is uniformly tight for $a\in ]0,1]$.
\end{itemize}
\end{theorem}
\begin{proof}
The mapping $\inf\limits_{\theta\in\Theta}\cR_{\rho_{p,a}}\big(\hat{F}_{n,\theta}\big) - \inf\limits_{\theta\in\Theta}\cR_{\rho_{p,a}}\big(F_{\theta}\big)$ is a well-defined random variable for $a\in ]0,1]$ due to Lemma \ref{Rueckfuehrung} along with Lemma \ref{ReduktionCountable} and completeness of $\OFP$. Statement 2) may be concluded from statement 1) in the same way as Theorem \ref{tightness} was derived from Theorem \ref{exponential bound}. Hence statement 1) is left to show.
\par 
Let $\overline{\Theta}\subseteq\Theta$ be from (A 3'). By Lemma \ref{ReduktionCountable} together with Lemma \ref{Rueckfuehrung} we have 
\begin{align}
\label{stern}
\MP\Big(\Big\{\big|\inf_{\theta\in\Theta}\cR_{\rho_{p,a}}\big(\hat{F}_{n,\theta}\big) - \inf_{\theta\in\Theta}\cR_{\rho_{p,a}}\big(F_{\theta}\big)\big|\geq\varepsilon\Big\}\Big)
\leq \MP\big(D_{n,\varepsilon,a}^{\overline{\Theta}}\big) + \MP\big(\overline{D}_{n,\varepsilon,p,a}^{\overline{\Theta}}\big)
\end{align}
for $n\in\NNN$, $\varepsilon > 0$, $a\in ]0,1]$,
where the sets $D_{n,\varepsilon,a}^{\overline{\Theta}}$ and $\overline{D}_{n,\varepsilon,p,a}^{\overline{\Theta}}$ are defined as in Lemma \ref{Rueckfuehrung}. 
The inequality $2^{p+2} J(\FFF^{\Theta},\xi,1/2^{p+4})\geq J(\FFF^{\Theta},\xi,1/4)$ holds (see \cite[Lemma 3.5.3]{GineNickl2016}). Moreover, $\|\xi_{p}\|_{\MP^{Z},2}\geq \|\xi_{p}\|_{\MP^{Z},2}^{1/p}\geq \|\xi\|_{\MP^{Z},2}$ due to Jensen's inequality. 
Then in view of Lemma \ref{auxiliary I} it is easy to check that the requirements of Theorem \ref{exponential bound} are met for both classes $\FFF^{\Theta}$ and $\FFF^{\Theta,p}$. Then statement 1) follows immediately from \eqref{stern} after application of Theorem \ref{exponential bound} separately to $\FFF^{\Theta}$ and $\FFF^{\Theta,p}$.
\end{proof}
\begin{remark}
\label{simplifications III}
Let us discuss upper estimations of the probabilities of the sets $\Omega\setminus B_{n}^{\xi}$ and $\Omega\setminus B_{n}^{\xi_{p}}$. 
\begin{itemize}
\item [1)] If the function $G$ is uniformly bounded by some positive constant $L$, then we may choose $\xi\equiv L$. Then $\Omega\setminus B_{n}^{\xi} = \Omega\setminus B_{n}^{\xi_{p}} = \emptyset$ for every $n\in\NNN$.
\item [2)] If $\xi$ is $\MP^{Z}$-integrable of order $4(p+1)$, then we have $$
\MP(\Omega\setminus B_{n}^{\xi}) \leq {\vari[\xi(Z_{1})^{2}]\over  n~\EEE[\xi(Z_{1})^{2}]^{2}}\quad\mbox{and}\quad 
\MP(\Omega\setminus B_{n}^{\xi_{p}}) \leq {\vari[\xi_{p}(Z_{1})^{2}]\over  n~\EEE[\xi_{p}(Z_{1})^{2}]^{2}}\quad\mbox{for}~n\in\NNN
$$ 
due to Chebychev's inequality. We may even obtain exponential bounds 
\begin{align*}
\MP(\Omega\setminus B_{n}^{\zeta}) \leq \exp\left(- n \EEE[\zeta(Z_{1})^{2}]^{2}/(8 \delta_{\zeta}^{2})\right)\vee \exp\left(- n~\EEE[\zeta(Z_{1})^{2}]/(4\delta_{\zeta})\right)
%&
%\MP(\Omega\setminus B_{n}^{\xi}) \leq \exp\left(- n \EEE[\xi(Z_{1})^{2}]^{2}/(8 \delta^{2})\right)\vee \exp\left(- n~\EEE[\xi(Z_{1})^{2}]/(4\delta)\right)\\
%&
%\MP(\Omega\setminus B_{n}^{\xi_{p}}) \leq \exp\left(- n \EEE[\xi_{p}(Z_{1})^{2}]^{2}/(8 \delta_{p}^{2})\right)\vee \exp\left(- n~\EEE[\xi_{p}(Z_{1})^{2}]/(4\delta_{p})\right)
\end{align*}
%In the same way as in Remark \ref{Simplifications}, 3), we may even obtain exponential bounds for these probabilities 
for $\zeta\in\{\xi,\xi_{p}\}$ if $\EEE\big[\exp\big(\lambda \xi_{p}(Z_{1})^{2}\big)\big] < \infty$ for some $\lambda > 0$. Note that this property is satisfied iff $\EEE\big[\exp\big(\lambda \xi(Z_{1})^{2(p+1)}\big)\big] < \infty$ for some $\lambda > 0$. The constant $\delta_{\zeta}$ may be chosen arbitrarily by $\delta_{\zeta}\geq M(\zeta^{2})$, where 
$M(\zeta^{2})$ is defined in the same way as in Remark \ref{Simplifications}, 3).
\end{itemize}
\end{remark}
\begin{remark}
\label{Diskussion I}
Theorem \ref{errorRatesUpperSemideviations} may be simplified if the objective $G$ satisfies condition (H), or if it has representation (PL). This may be seen immediately in view of Proposition \ref{Hoelder-Bedingung auxiliary}, or Proposition \ref{startingpoint} along with Remark \ref{Bemerkung I}. In addition we may invoke more explicit upper estimates for the term $J(\FFF^{\Theta},\xi,1/4)$ provided by Proposition \ref{Hoelder-Bedingung auxiliary} and Proposition \ref{startingpoint}.
\par
According to Remark \ref{simplifications III} the error estimates in Theorem \ref{errorRatesUpperSemideviations} may be further improved if the mapping $G$ is bounded. In this situation a 
version has been shown in \cite{BartlTangpi2020} for bounded $G$ having the form $G(\theta,z) := W_{0}(z) + \langle\theta, \underline{W}\rangle$, where $W_{0}$ and $\underline{W}$ are fixed Borel measurable mappings, and $\langle\cdot,\cdot\rangle$ stands for the standard scalar product on $\RRR^{m}$. However, the bounds for deviation probabilities derived in \cite{BartlTangpi2020} are described in unknown universal constant. In contrast combining Theorem \ref{errorRatesUpperSemideviations} with Proposition \ref{Hoelder-Bedingung auxiliary} we may provide more explicit bounds.
\par
The statement on uniform tightness in Theorem \ref{errorRatesUpperSemideviations} has been already shown in \cite{DentchevaEtAl2017} for $p > 1$ under (H) with $\beta = 1$.
\end{remark}
%%%%%%%%%%%%%%%%%%%%%%%%%%%%%%%%%%%%%%%%%%%%%%%%%%%%%%%%%%%%%%%%%%%%%%%%%%%%%%%%%%%%%%%%%%%%%%%%%%%%%%%%%%%%%%%%%%%%%%%%%%%%%%%%%%%%%%%%%%%%%%%%%%%%%%%%%%%%%%%%%%%%%%%%%%%%%%%%%%%%%%%%%%%%%%%%%%%%%%%%%%%%%%%%%%%%%%%%%%%%%%%%%%%%%%%%%%%%%%%%%%%%%%%%%%%%%%%%%%%%%%%%%%%%%%%%%%%%%%%%%%%%%%%%
\section{Error estimates under divergence risk measures}
\label{generalized divergence risk measures}
We want to study the risk averse stochastic program (\ref{optimization risk averse}), where  
we shall focus on $\rho$ being
a divergence measure. For introduction, let us consider a lower semicontinuous convex mapping $\Phi: [0,\infty[\rightarrow [0,\infty]$ satisfying 
$\Phi(0) < \infty$, $\Phi(x_{0}) < \infty$ for some $x_{0} > 1,$ $\inf_{x\geq 0}\Phi(x) = 0,$ and the growth condition
$\lim_{x\to\infty}\frac{\Phi(x)}{x} = \infty.$ Its Fenchel-Legendre transform
$$
\Phi^{*}:\R\rightarrow \R\cup\{\infty\},~y\mapsto\sup_{x\geq 0}~\big(xy - \Phi(x)\big)
$$
is a finite nondecreasing convex function whose restriction $\Phi^{*}\bigr |_{[0,\infty[}$ to 
$[0,\infty[$ is a finite Young function, i.e.  a continuous nondecreasing and unbounded real-valued mapping with $\Phi^{*}(0) = 0$ (cf. \cite[Lemma A.1]{BelomestnyKraetschmer2016}). Note also that the right-sided derivative $\Phi^{*'}$ of $\Phi^{*}$ is nonnegative and nondecreasing. We shall use  
$H^{\Phi^{*}}$ to denote the \textit{Orlicz heart} w.r.t. $\Phi^{*}\bigr |_{[0,\infty[}$ defined to mean the set of all random variables $X$ on $\OFP$ satisfying $\ex [\,\Phi^{*}(c|X|)\,]<\infty$ for all $c > 0$. As in the previous Section \ref{upper semideviations}  we identify random variables which differ on $\MP$-null sets only.
\par
%$$
%H^{\Phi^{*}}:=H^{\Phi^{*}}(\Omega,\cF,\pr)=\big\{X\in L^0\,:\,\ex [\,\Phi^{*}(c|X|)\,]<\infty~
%\mbox{ for all $c>0$}\big\}.
%$$
The Orlicz heart is known to be a vector space enclosing all $\MP$-essentially bounded random variables. Moreover,  
%Banach space when endowed with the \textit{Luxemburg norm}
% $$
%    \|X\|_{\Phi^{*}} := \inf\left\{\lambda > 0\,:\,\ex [\,\Phi^{*}(|X|/\lambda)\,]\leq 1\right\}.
%$$
by Jensen's inequality all members of $H^{\Phi^{*}}$ are $\MP$-integrable.
%$\subseteq L^{1}.$ In addition, we may see that $L^{\infty}$ is a linear subspace of $H^{\Phi^{*}},$ which is dense w.r.t. $\|\cdot\|_{\Phi^{*}}$ (see Theorem 2.1.14 in \cite{EdgarSucheston1992}). 
For more on Orlicz hearts w.r.t. to Young functions the reader may consult \cite{EdgarSucheston1992}.
\medskip

We can define the following mapping
$$
\rho^{\Phi}(X)=\sup_{\overline{\MP}\in\cP_{\Phi}}\left(\ex_{\overline{\MP}}\left[X\right] - 
\ex\left[\Phi\left(\frac{d\overline{\MP}}{d\pr}\right)\right]\right)
$$
for all \(X\in H^{\Phi^{*}},\)  where $\cP_{\Phi},$ denotes the set of all probability measures $\overline{\MP}$ which are absolutely continuous w.r.t. $\pr$  such that $\Phi\left(\frac{d\overline{\MP}}{d\pr}\right)$ is $\pr-$integrable. Note that 
$\frac{d\overline{\MP}}{d\pr}~ X$ is $\pr-$integrable for every $\overline{\MP}\in \cP_{\Phi}$ and any 
$X\in H^{\Phi^{*}}$ due to Young's inequality. We shall call $\rho^{\Phi}$ the \textit{divergence risk measure w.r.t. $\Phi$}. 
%It obviously satisfies 
%$$
%\rho^{\Phi}(Y)\geq\rho^{\Phi}(X)\quad\mbox{if}~\EEE[(Y-x)^{+}]\geq\EEE[(X - x)^{+}]~\mbox{holds for}~X,Y\in H^{\Phi^{*}}~\mbox{and every}~x\in\RRR.
%$$
%By well-known characterizations of the increasing convex order (see e.g. \cite[Theorem 2.57]{FoellmerSchied2011}) this means that $\rho^{\Phi}$ is indeed nondecreasing w.r.t. the increasing convex order. 
\medskip

Ben-Tal and Teboulle (\cite{Ben-TalTeboulle1987}, 
\cite{Ben-TalTeboulle2007}) discovered another more convenient representation. 
%We try to provide a representation which might be more convenient for calculations and for issues like optimal stopping. 
%For preparation, let $\textrm{dom}(\Phi)$ denote the effective domain of $\Phi$ which is nonvoid by assumption. Setting $x_{\Phi} \dot= \sup\textrm{dom}(\Phi)$ we also know $]0,x_{\Phi}[\subseteq\text{dom}(\Phi)$. We shall assume
%\begin{equation}
%\label{domain-Bedingung}
%x_{\Phi} > 1.
%\end{equation}
It reads as follows (see \cite{BelomestnyKraetschmer2016}).
\begin{theorem}
\label{optimized certainty equivalent}
The divergence risk measure $\rho^{\Phi}$ w.r.t. $\Phi$ satisfies the following representation
\begin{eqnarray*}
\rho^{\Phi}(X) 
= 
\inf_{x\in\R}\ex\left[\Phi^{*}(X + x) - x\right]
\quad\mbox{for all}~X\in H^{\Phi^{*}}.
\end{eqnarray*}
\end{theorem}
The representation in Theorem \ref{optimized certainty equivalent} is also known as 
the \textit{optimized certainty equivalent w.r.t. $\Phi^{*}$}. As optimized certainty equivalent the divergence measure $\rho^{\Phi}$ may be seen directly to be nondecreasing w.r.t. the increasing convex order. Theorem \ref{optimized certainty equivalent} also shows that $\rho^{\Phi}$ is distribution-invariant. In particular, we may define the functional $\cR_{\rho^{\Phi}}$ associated with $\rho^{\Phi}$ on the set of all distribution functions of the random variables from $H^{\Phi^{*}}$. 
Throughout this section we focus on the following specialization of optimization problem (\ref{optimization risk averse})
\begin{equation}
\label{optimization general II}
\inf_{\theta\in\Theta}\cR_{\rho^{\Phi}}\big(F_{\theta}\big),
\end{equation}
where $F_{\theta}$ stands for the distribution function of $G(\theta,Z)$ for 
$\theta\in\Theta$. 
\par
The SAA (\ref{SAA risk averse}) of (\ref{optimization general II}) reads as follows. 
\begin{equation}
\label{optimization approximativ}
\inf_{\theta\in\Theta}\cR_{\rho^{\Phi}}\big(\hat{F}_{n,\theta}\big) 
= \inf_{\theta\in\Theta}\inf_{x\in\RRR}\Big(\frac{1}{n}\sum_{i=1}^{n}\Phi^{*}\big( G(\theta, Z_{i}) +x\big) - x\Big)\quad(n\in\NNN).
\end{equation}
We shall strengthen condition (A 2) to the following property. 
\smallskip

\noindent
\begin{itemize}
\item [(A 2')] There exists some positive envelope $\xi$ of $\FFF^{\Theta}$ satisfying $\xi(Z_{1})\in\cH^{\Phi^{*}}$.
\end{itemize}
\smallskip

\noindent
Note that (A 2') together with (A 1) implies that $G(\theta,Z_{1})$ belongs to $\cH^{\Phi^{*}}$ for every $\theta\in\Theta$ so that the genuine optimization problem \eqref{optimization general II} is well-defined.
\par
We are mainly interested in deviation probabilities \eqref{Kreuz} w.r.t. $\cR_{\rho^{\Phi}}$. Representation \eqref{optimization approximativ} along with Theorem \ref{optimized certainty equivalent} suggests to apply Theorem \ref{exponential bound} to the SAA of 
$$
\inf_{(\theta,x)\in\Theta\times\RRR}\EEE\big[G_{\Phi}\big((\theta,x),Z_{1}\big)\big],
$$
where
\begin{equation}
\label{new objective}
G_{\Phi}: (\Theta\times\RRR)\times\RRR^{d}\rightarrow\RRR, \big((\theta,x),z\big)\mapsto \Phi^{*}\big(G(\theta,z) + x\big) - x.
\end{equation}
Unfortunately, the application is not immediate because the parameter space is not totally bounded w.r.t. the Euclidean metric on $\RRR^{d}$. So a kind of compactification is needed, provided by the following result.
For preparation let us consider any mapping $\xi$ as in (A 2') and let $x_{0} >1$ be from the effective domain of $\Phi$. Then we introduce for $\delta > 0$ the following real numbers
\begin{eqnarray}
&&\label{Eingrenzung1}
x_{l}(x_{0},\xi,\delta) := - \Phi(0) - \delta - \EEE\big[\Phi^{*}\big(\xi(Z_{1})\big)\big]\\
&&\label{Eingrenzung2} 
x_{u}(x_{0},\xi,\delta) := \frac{\Phi(x_{0}) + (1 + x_{0})\delta + \EEE\big[\Phi^{*}\big(\xi(Z_{1})\big)\big] + x_{0} \EEE[\xi(Z_{1})]}{x_{0} - 1} + \Phi(0).
\end{eqnarray}
Note that by (A 2') along with Jensen's inequality the mapping $\xi$ is $\MP^{Z}$-integrable. For abbreviation we set, using notations \eqref{Eingrenzung1} as well as \eqref{Eingrenzung2}
\begin{equation}
\label{Eingrenzung3}
I_{x_{0},\xi,\delta} :=  [x_{l}(x_{0},\xi,\delta),x_{u}(x_{0},\xi,\delta)].
\end{equation}
\begin{proposition}
\label{compactification}
Let (A 1), (A 2') be fulfilled. Furthermore, for $\delta > 0$ and $n\in\NNN$ the set 
$A_{n,\delta}^{\xi}\in\cF$ is defined to consist of all $\omega\in\Omega$ satisfying
$$
\frac{1}{n}\sum_{j=1}^{n}\xi\big(Z_{j}(\omega)\big)\leq \EEE\big[\xi(Z_{1})\big] + \delta,~\frac{1}{n}\sum_{j=1}^{n}\Phi^{*}\Big(\xi\big(Z_{j}(\omega)\big)\Big)\leq  \EEE\big[\Phi^{*}\big(\xi(Z_{1})\big)\big] + \delta.
$$
If $G(\cdot,z)$ is lower semicontinuous for $z\in\RRR^{d}$, then optimal values of \eqref{optimization approximativ} and \eqref{optimization general II} are always finite, and, using notations \eqref{new objective}, \eqref{Eingrenzung3}, if $\omega\in A^{\xi}_{n,\delta}$, then
\begin{align*}
&
\inf_{\theta\in\Theta}\cR_{\rho^{\Phi}}\big(\hat{F}_{n,\theta}\big) - \inf_{\theta\in\Theta}\cR_{\rho^{\Phi}}(F_{\theta})\\
&
= 
\inf_{(\theta,x)\in\Theta\times I_{x_{0},\xi,\delta}}\frac{1}{n}\sum_{j=1}^{n}G_{\Phi}\big((\theta,x),Z_{j}\big) - \inf_{(\theta,x)\in\Theta\times I_{x_{0},\xi,\delta}}\EEE\big[G_{\Phi}\big((\theta,x),Z_{1}\big)\big].
\end{align*}
\end{proposition}
The proof of Proposition \ref{compactification} may be found in Subsection \ref{Beweis compactification}.
\medskip

Now in view of Proposition \ref{compactification}, we may derive the desired deviation probabilities by applying Theorem \ref{exponential bound} to the function classes of the following type
\begin{equation}
\label{neue Funktionsklassen}
\FFF^{\Theta}_{\Phi,I} :=\big\{G_{\Phi}\big((\theta,x),\cdot\big)\mid (\theta,x)\in\Theta\times I\big\}\quad(I\subseteq\RRR~\mbox{compact interval}).
\end{equation}
However, we want to formulate the requirement by means of the terms $J(\FFF^{\Theta},C_{\FFF^{\Theta}},\delta)$ associated with the genuine objective $G$ instead of the terms $J(\FFF^{\Theta}_{\Phi,I},C_{\FFF^{\Theta}_{\Phi,I}},\delta)$. The relationship between these terms is the subject of the following auxiliary result.
\begin{lemma}
\label{relationship}
Let $I\subseteq\RRR$ be a nondegenerated compact interval fulfilling the property $\sup I = |\inf I|\vee |\sup I| > 0$, and let $\Phi^{*'}_{+}$ denote the right-sided derivative of $\Phi^{*}$. If $\xi$ is a square $\MP^{Z}$-integrable positive envelope of $\FFF^{\Theta}$, then 
$$
C_{\FFF^{\Theta}_{\Phi,I}} := 2 \big[\Phi^{*'}_{+}\big(\xi + \sup I\big) + 1] \sqrt{\xi^{2} + (\sup I)^{2}}
$$
is a positive envelope of $\FFF^{\Theta}_{\Phi,I}$ satisfying 
\begin{align*}
J(\FFF^{\Theta}_{\Phi,I},C_{\FFF^{\Theta}_{\Phi,I}},\delta)\leq \sqrt{2}~ J(\FFF^{\Theta},\xi,\delta) + 4\delta\sqrt{\ln(1/\delta)} + \sqrt{2\ln(2)}~\delta\quad\mbox{for}~\delta \in ]0,\exp(-1)].
\end{align*}
\end{lemma}
The proof may be found in Subsection \ref{Beweis compactification}.
\medskip

Next, we want to find an analogue of (A 3) for the auxiliary goal $G_{\Phi}$ but in terms of the genuine one $G$. It is the following one.
\begin{itemize}
\item [(A 3'')]There exist some at most countable subset $\overline{\Theta}\subseteq\Theta$ and $(\MP^{Z})^{n}$-null sets $N_{n}$ $(n\in\NNN)$ such that 
$$
\inf_{\vartheta\in\overline{\Theta}}\EEE[|G(\vartheta,Z_{1})- G(\theta,Z_{1})|]= \inf_{\vartheta\in\overline{\Theta}}\max_{j\in\{1,\ldots,n\}}\big|G(\theta,z_{j}) - G(\vartheta,z_{j})\big| = 0
$$
for $n\in\NNN,\theta\in\Theta$ and $(z_{1},\ldots,z_{n})\in\RRR^{d n}\setminus N_{n}$.
\end{itemize}
\begin{remark}
\label{Bemerkung II}
Criteria for (A 3'') in the cases that $G$ satisfies (H) or has representation (PL) carry over directly from Remark \ref{Bemerkung I}. This is because (A 3'') is implied by (A 3').
\end{remark}
\begin{lemma}
\label{countable parameter subsets}
Let (A 1), (A 2') and (A 3'') be fulfilled, and let $I\subseteq\RRR$ denote a nondegenerated interval. Then with the at most countable subset $\overline{\Theta}\subseteq\Theta$ and the $(\MP^{Z})^{n}$-null sets $N_{n}$ $(n\in\NNN)$ from (A 3'') it holds 
\begin{align*}
&
\inf_{(\vartheta,y)\in\overline{\Theta}\times I\cap\mathbb{Q}}\big|\EEE\big[G_{\Phi}\big((\vartheta,y),Z_{1}\big)\big] - \EEE\big[G_{\Phi}\big((\theta,x),Z_{1}\big)]\big|\\
&= \inf_{(\vartheta,y)\in\overline{\Theta}\times I\cap\mathbb{Q}}\max_{j\in\{1,\ldots,n\}}\big|G_{\Phi}\big((\theta,y),z_{j}\big) - G_{\Phi}\big((\vartheta,x),z_{j}\big)\big| = 0
\end{align*}
for $n\in\NNN$, $\theta\in\Theta$, $x\in I$ and $(z_{1},\ldots,z_{n})\in\RRR^{d n}\setminus N_{n}$.
\end{lemma}
The proof is postponed to Subsection \ref{Beweis compactification}.
\par
Putting together Proposition \ref{compactification} and Lemmata \ref{relationship}, \ref{countable parameter subsets}, we end up with the following result on the deviation probabilities. Recall notations \eqref{Eingrenzung2}, and $\Phi^{*'}_{+}$ for the right-sided derivative of $\Phi^{*}$.
\begin{theorem}
\label{errors divergence risk measures}
Let (A 1), (A 2'), (A 3'') be fulfilled. Using notation \eqref{Eingrenzung2} the Borel measurable mapping $\xi$ from (A 2') is assumed to satisfy the property that the mapping $\xi_{x_{0},\xi,\delta} := [\Phi^{*'}_{+}\big(\xi + x_{u}(x_{0},\xi,\delta)\big) + 1]\sqrt{\xi^{2} + x_{u}(x_{0},\xi,\delta)^{2}}$ is square $\MP^{Z}$-integrable for some $x_{0} \in ]1,2[$ from the effective domain of $\Phi$ and $\delta > 0$. 
If $G(\cdot,z)$ is lower semicontinuous for $z\in\RRR^{d}$, and if $J(\FFF^{\Theta},\xi,1/2)$ is finite, then the following statements are true.
\begin{itemize}
\item [1)] For $\varepsilon, t > 0$ and $n\in\NNN$ with $n\geq 2 \|\xi_{x_{0},\xi,\delta}\|_{\MP^{Z},2}^{2}$ the inequality 
\begin{align*}
&
\MP\left(\left\{\Big|\inf\limits_{\theta\in\Theta}~\cR_{\rho^{\Phi}}\big(\hat{F}_{n,\theta}\big) - \inf\limits_{\theta\in\Theta}\cR_{\rho^{\Phi}}\big(F_{\theta}\big)\Big|\geq\varepsilon\right\}\right)\\
\\
&\leq 
\exp\left(\frac{-t~\sqrt{n}\varepsilon}{16 (t + 1) \|\xi_{x_{0},\xi,\delta}\|_{\MP^{Z},2}}\cdot\frak{f}_{n}(t)\right)
+ \MP\big(\Omega\setminus A_{n,\delta}^{\xi}\big) + \MP\big(\Omega\setminus B_{n}^{2 \xi_{x_{0},\xi,\delta}}\big),
\end{align*}
holds if $\varepsilon > \frac{\|\xi_{x_{0},\xi,\delta}\|_{\MP^{Z},2}}{\sqrt{n}}\big[2 + 32 (t + 1) \big(4 J(\FFF^{\Theta},\xi,1/4) + 5\sqrt{\ln(2)}\big)\big]$. 
Here $A^{\xi}_{n,\delta}$ is as in the display of Proposition \ref{compactification}, and $B_{n}^{2 \xi_{x_{0},\xi,\delta}}$ is defined according to \eqref{Restmenge}.
\item [2)] The sequence 
$
\big(\sqrt{n}\big[\inf\limits_{\theta\in\Theta}~\cR_{\rho^{\Phi}}(\hat{F}_{n,\theta}) - \inf\limits_{\theta\in\Theta}\cR_{\rho^{\Phi}}(F_{\theta})\big]\big)_{n\in\NNN}
$
is a uniformly tight sequence of random variables.
\end{itemize} 
\end{theorem}
\begin{proof}
Let $\overline{\Theta}\subseteq\Theta$ be from (A 3''). Combining Theorem \ref{optimized certainty equivalent} and \eqref{optimization approximativ} with Lemma \ref{countable parameter subsets}, we may observe 
\begin{align*}
&
\inf_{\theta\in\Theta}\cR_{\rho^{\Phi}}\big(F_{\theta}\big) = \inf_{(\theta,x)\in\overline{\Theta}\times\mathbb{Q}}\EEE\big[G_{\Phi}\big((\theta,x),Z_{1}\big)\big],\\
&
\inf_{\theta\in\Theta}~\cR_{\rho^{\Phi}}\big(\hat{F}_{n,\theta}\big) = \inf_{(\theta,x)\in\overline{\Theta}\times\mathbb{Q}}\frac{1}{n}\sum_{j=1}^{n}G_{\Phi}\big((\theta,x),Z_{j}\big)\quad\MP-\mbox{a.s.}\quad\mbox{for}~n\in\NNN.
\end{align*}
In particular, taking Proposition \ref{compactification} and completeness of $\OFP$ into account, 
$$
\inf_{\theta\in\Theta}~\cR_{\rho^{\Phi}}\big(\hat{F}_{n,\theta}\big) - \inf_{\theta\in\Theta}\cR_{\rho^{\Phi}}\big(F_{\theta}\big)\quad\mbox{is a random variable for}~n\in\NNN.
$$
Let $I_{x_{0},\xi,\delta}$ denote the interval defined in \eqref{Eingrenzung3}. By Proposition \ref{compactification} along with Lemma \ref{countable parameter subsets} we have 
\begin{align*}
&
\inf_{\theta\in\Theta}\cR_{\rho^{\Phi}}\big(F_{\theta}\big) 
= 
\inf_{(\theta,x)\in\overline{\Theta}\times I_{x_{0},\xi,\delta}\cap\mathbb{Q}}\EEE\big[G_{\Phi}\big((\theta,x),Z_{1}\big)\big],\\
&
\inf_{\theta\in\Theta}~\cR_{\rho^{\Phi}}\big(\hat{F}_{n,\theta}\big)(\omega) = 
\inf_{(\theta,x)\in\overline{\Theta}\times I_{x_{0},\xi,\delta}\cap\mathbb{Q}}\frac{1}{n}\sum_{j=1}^{n}G_{\Phi}\big((\theta,x),Z_{j}(\omega)\big)\quad\mbox{for}~n\in\NNN, \omega\in A_{n,\delta}^{\xi}. 
\end{align*} 
Finally, note that $\sup I_{x_{0},\xi,\delta} = |\sup I_{x_{0},\xi,\delta}|\vee |\inf I_{x_{0},\xi,\delta}| > 0$ holds. Now, we may apply Theorems \ref{exponential bound},~\ref{tightness} to the function class $\FFF^{\Theta}_{\Phi,I_{x_{0},\xi,\delta}}$, as defined in \eqref{neue Funktionsklassen}. Then in view of Lemma \ref{relationship} we may derive easily the statements of Theorem \ref{errors divergence risk measures}. 
\end{proof}
\begin{remark}
\label{simplifications II}
Let us point out some simplifications of Theorem \ref{errors divergence risk measures}.
\begin{itemize}
\item [1)] If the function $G$ is uniformly bounded by some positive constant $L$, then we may choose $\xi\equiv L$. Then $\Omega\setminus A_{n,\delta}^{\xi} = \Omega\setminus B_{n}^{2 \xi_{x_{0},\xi,\delta}} = \emptyset$ for every $n\in\NNN$.
\item [2)] Set $\zeta_{1} := \xi, \zeta_{2} := \Phi^{*}\circ\xi, \zeta_{3} := \xi_{x_{0},\xi,\delta}$. By Chebychev's inequality we have 
$$
\MP\big(\Omega\setminus A_{n,\delta}^{\xi}\big) + \MP\big(\Omega\setminus B_{n}^{2 \xi_{x_{0},\xi,\delta}}\big) \leq \sum_{i=1}^{3} {\vari[\zeta_{i}(Z_{1})^{2}]\over  n~\EEE[\zeta_{i}(Z_{1})^{2}]^{2}}\quad\mbox{for}~n\in\NNN
$$ 
if 
$\xi_{x_{0},\xi,\delta}$ is integrable of order $4$. Analogously to Remark \ref{Simplifications}, 3), we may even obtain exponential bounds 
\begin{align*}
&
\MP\big(\Omega\setminus A_{n,\delta}^{\xi}\big) + \MP\big(\Omega\setminus B_{n}^{2 \xi_{x_{0},\xi,\delta}}\big)\\
&\leq \sum_{i=1}^{3}\exp\left(- n \EEE[\zeta_{i}(Z_{1})^{2}]^{2}/(8 \delta_{\zeta_{i}}^{2})\right)\vee \exp\left(- n~\EEE[\zeta_{i}(Z_{1})^{2}]/(4\delta_{\zeta_{i}})\right)\quad\mbox{for}~n\in\NNN
\end{align*}
if $\EEE\big[\exp\big(\lambda \xi_{x_{0},\xi,\delta}(Z_{1})^{2}\big)\big]$ is finite for some $\lambda > 0$. With $M(\zeta_{i}^{2})$ as in Remark \ref{Simplifications}, 3), the inequalities hold for any $\delta_{\zeta_{i}} \geq M(\zeta_{i}^{2})$ ($i=1,2,3$).
%$$
%\MP\big(\Omega\setminus \big[A_{n,\delta}^{\xi}\cap B_{n}^{4 \xi_{x_{0},\xi,\delta}}\big]\big) \leq \frac{\vari[\xi] + {\vari}\big[\Phi^{*}\big(\xi(Z_{1})\big)\big]}{\delta \sqrt{n}}~+~\frac{4 \vari\big[\xi_{x_{0},\xi,\delta}^{2}\big]}{\EEE\big[\xi_{x_{0},\xi,\delta}^{2}\big]\sqrt{n}} 
%{\vari}\big[\Phi^{*}\big(\xi_{1}(Z_{1})+ \overline{x}_{\xi_{1},x_{0}}\big)\big]\over \sqrt{n}}.
%$$ 
%\item [3)] {\color{magenta}Weitere Bemerkungen mit light tail und Bernstein-Ungleichung?}
%The upper estimate of the probability $\MP(B_{n})$ in Theorem \ref{second bound} may be further improved if the random variable $\exp\Big(\lambda\cdot\Phi^{*}\big(\xi_{1}(Z_{1}) + \overline{x}_{\xi_{1},x_{0}}\big)\Big)$ is integrable for some $\lambda > 0$ (\cite[Theorem 5.1]{BuldyginKozachenko2000}).
\end{itemize}
\end{remark}
\begin{remark}
Drawing on Proposition \ref{Hoelder-Bedingung auxiliary}, or Proposition \ref{startingpoint} along with Remark \ref{Bemerkung II}, we may simplify directly Theorem \ref{errors divergence risk measures} in the cases that $G$ fulfills property (H), or has representation (PL). Moreover, Theorem \ref{errors divergence risk measures} may be improved in the way that the results provide explicit upper bounds for the involved term $J(\FFF^{\Theta},\xi,1/4)$. 
\par 
In the simplified situation of bounded $G$ error estimates have been developped in \cite{BartlTangpi2020} for  linear $G$ as already described in Remark \ref{Diskussion I}. As in this remark we want to emphasize again that universal unknown constants are involved in the bounds from \cite{BartlTangpi2020}. This shortcoming may be avoided by Theorem \ref{errors divergence risk measures} for this special type of objective $G$, just by using  Proposition \ref{Hoelder-Bedingung auxiliary}.
\end{remark}
Let us look at the specialization of Theorem \ref{errors divergence risk measures} in the important case that $\rho^{\Phi}$ is the \textit{Average Value at Risk}, also known as the \textit{Expected Shortfall}.
\begin{example}
\label{AVaR}
Let $\Phi$ be defined by $\Phi_{\alpha}(x) := 0$ for $x\leq 1/(1-\alpha)$ for some $\alpha\in ]0,1[$, and $\Phi(x) := \infty$ if $x > 1/(1-\alpha)$. Then $\Phi^{*}_{\alpha}(y) = y^{+}/(1-\alpha)$ for $y\in\R$. In particular $H^{\Phi^{*}}$ coincides with $L^{1}$, and we may recognize $\cR_{\rho^{\Phi}}$ as the so called \textit{Average Value at Risk} w.r.t. $\alpha$ (e.g. \cite{FoellmerSchied2011}, \cite{ShapiroEtAl}), i.e.
\begin{align*}
\cR_{\rho^{\Phi}}(F) 
&
= 
\frac{1}{1-\alpha}~\int_{\Flinks(\alpha)}^{1}\eins_{]0,1[}(u)~\Flinks(u)~du
\\ 
&
= \inf_{x\in\R}\left(\int_{0}^{1}\eins_{]0,1[}(u)~\frac{(\Flinks(u) + x)^{+}}{1-\alpha}~du - x\right)
\end{align*}
(see e.g. \cite{KainaRueschendorf2007}), where $\Flinks$ denotes the left-continuous quantile function of $F$. In this situation we have the following specifications of some particular assumptions in Theorem \ref{errors divergence risk measures}.
\begin{itemize}
\item (A 2) and (A 2'') are equivalent.
\item If $\xi:\RRR^{d}\rightarrow\RRR$ is any strictly positive square $\MP^{Z}$-integrable mapping, then 
$$
[\Phi^{*'}_{+}(\xi + a) + 1]\sqrt{\xi^{2} + a^{2}} = (2 - \alpha)\sqrt{\xi^{2} + a^{2}}/(1- \alpha)
$$ 
is already square $\MP^{Z}$-integrable for every $a > 0$.
\item The sets $A_{n,\delta}^{\xi}, B_{n}^{2\xi_{x_{0},\xi,\delta}}$ from Theorem \ref{errors divergence risk measures} may be simplified as follows
\begin{align*}
&
A_{n,\delta}^{\xi} = \Big\{\frac{1}{n}\sum_{j=1}^{n}\xi(Z_{j})\leq \EEE[\xi(Z_{1})] + (1-\alpha)\delta,\Big\},\\
& 
B_{n}^{2\xi_{x_{0},\xi,\delta}} = \Big\{\frac{1}{n}\sum_{j=1}^{n}\xi(Z_{j})^{2}\leq 2\EEE[\xi(Z_{1})^{2}] + x_{u}(x_{0},\xi,\delta)^{2}\Big\}
\end{align*}
where $x_{u}(x_{0},\xi,\delta)$ is as in \eqref{Eingrenzung2}.
\item Under condition (H) with $\beta = 1$ the uniform tightness result in Theorem \ref{errors divergence risk measures} is already known from \cite{GuiguesKraetschmerShapiro2018}.
\end{itemize}
\end{example}
\section{Proofs}
\label{Beweise}
%%%%%%%%%%%%%%%%%%%%%%%%%%%%%%%%%%%%%%%%%%%%%%%%%%%%%%%%%%%%%%%%%%%%%%%%%%%%%%%%%%%%%%%%%%%%%%%%%%%%%%%%%%%%%%%%%%%%%%%%%%%%%%%%%%%%%%%%%%%%%%%%%%%%%%%%%%%%%%%%%%%%%%%%%%%%%%%%%%%%%%%%
\subsection{Proof of Theorem \ref{exponential bound}}
\label{zweite Schranke}
The main tool for the proof of Theorem \ref{exponential bound} is Bousquet's version of Talagrand's concentration inequalities. We shall repeat them first, tailored to our situation, for the convenience of the reader (see Theorem 3.3.9 in \cite{GineNickl2016}).
\begin{theorem}
\label{TalagrandBousquet}
Let $\FFF$ be some at most countable set of centered $\MP^{Z}$-integrable functions which is uniformly bounded by some positive constant $\overline{u}$. Assume that $\sigma^{2}\in ]0,\overline{u}]$ is an upper bound for the set $\{\vari(h)\mid h\in\FFF\}$. Then for every $n\in\NNN$ and any $\varepsilon > 0$ 
\begin{align*}
\MP\big(\big\{S_{n}\geq \EEE[S_{n}] + \varepsilon\big\}\big)
&\leq 
\exp\left(\frac{- 3\varepsilon~}{4 \overline{u}}~\ln\Big(1 + \frac{2\varepsilon \overline{u}}{6~\overline{u}~\EEE[S_{n}] + 3 n \sigma^{2}}\Big)\right)\\
&
\leq 
\exp\Big(\frac{-\varepsilon^{2}}{2\big(\overline{u}~2\EEE[S_{n}] + n\sigma^{2} + \overline{u}~\varepsilon/3\big)}\Big),
\end{align*}
where $S_{n} := \sup_{h\in\FFF}\big|\sum_{j=1}^{n}h(Z_{j})\big|$.
\end{theorem}
Now, we are prepared to show Theorem \ref{exponential bound}.
\medskip

\noindent
\textbf{Proof of Theorem \ref{exponential bound}:}\\[0.1cm]
As already discussed after introducing condition (A 3), we may replace in the optimization problems \eqref{optimization risk neutral}, \eqref{SAAriskneutral} the parameter space with the at most countable subset $\overline{\Theta}\subseteq\Theta$ from (A 3). Hence
\begin{align}
&\nonumber
\MP\Big(\Big\{\big|\inf\limits_{\theta\in\Theta}\frac{1}{n}\sum_{j=1}^{n}G(\theta,Z_{j}) - \inf\limits_{\theta\in\Theta}\EEE[G(\theta,Z_{1})]\big|\geq\varepsilon\Big\}~\cap~B_{n}^{\xi}\Big)\\
&\label{application Talagrand 2}
\leq 
\MP\Big(\Big\{\sup_{\theta\in\overline{\Theta}}\big|\frac{1}{n}\sum_{j=1}^{n}G(\theta,Z_{j}) - \EEE[G(\theta,Z_{1})]\big|\geq \varepsilon\Big\}~\cap~B_{n}^{\xi}\Big)\quad\mbox{for}~\varepsilon > 0.
\end{align}
By definition of $B_{n}^{\xi}$ we may observe for $\omega\in B_{n}^{\xi}$ and $j\in\{1,\ldots,n\}$
\begin{equation}
\label{firstupperbound}
\big|G\big(\theta,Z_{j}(\omega)\big)\big|\leq \big|\xi\big(Z_{j}(\omega)\big)\big|\leq w_{n} := \sqrt{ 2 n} \|\xi\|_{\MP^{Z,2}}.
\end{equation}
Then, setting $\phi_{n}(t) := (t\wedge w_{n})\vee (- w_{n})$ for $t \in\RRR$, we obtain
\begin{align*}
&
\big|\frac{1}{n}\sum_{j=1}^{n}G(\theta,Z_{j}(\omega)) - \EEE\big[G(\theta,Z_{1})\big]|\\
&\leq 
\big|\frac{1}{n}\sum_{j=1}^{n}\phi_{n}\big(G(\theta,Z_{j}(\omega))\big)  - \EEE\big[\phi_{n}\big(G(\theta,Z_{1})\big)\big]|
%&\quad  
+ \big|\EEE\big[\phi_{n}\big(G(\theta,Z_{1})\big) - \EEE[G(\theta,Z_{1})]\big]\big|
\end{align*}
for $\theta\in\Theta$, and $\omega\in B_{n}^{\xi}$. The function $\phi_{n}$ satisfies 
the following properties
\begin{equation}
\label{ersteEig}
|\phi_{n}(t) - \phi_{n}(s)|\leq |t - s|\quad\mbox{for}~t,s\in\RRR,
\end{equation}
and for any integrable random variable $W$
\begin{align}
\big|\EEE[\phi_{n}(W)] - \EEE[W] \big|
&\nonumber \leq 
\big|\EEE[(-w_{n} - W)\eins_{]-\infty,-w_{n}]}(W)]\big| + 
\big|\EEE[(W-w_{n})\eins_{[w_{n},\infty[}(W)]\big|\\ 
&
\label{zweiteEig} 
= 
\EEE[(-w_{n} - W)^{+}] + 
\EEE[(W-w_{n})^{+}] 
\end{align}
Invoking (A 2), we may conclude from \eqref{zweiteEig}
$$
\sup_{\theta\in\Theta}\big|\EEE\big[\phi_{n}\big(G(\theta,Z_{1})\big)] - \EEE[G(\theta,Z_{1})]\big|\leq 2 \EEE\big[\big(\xi(Z_{1}) - w_{n})^{+}\big] =: \delta_{n}.
$$
Furthermore by square integrability of $\xi(Z_{1})$
\begin{align*}
n \delta_{n} 
&=
2 n \int_{w_{n}}^{\infty}\MP\big(\{\xi(Z_{1}) > t\big\}\big)~dt\\
&= 
\frac{\sqrt{2n}}{\|\xi\|_{\MP^{Z},2}}~\int_{\sqrt{2n} \|\xi\|_{\MP^{Z},2}}^{\infty} \sqrt{2n} \|\xi\|_{\MP^{Z},2}~\MP\big(\{\xi(Z_{1}) > t\big\}\big)~dt\\
&\leq 
\frac{\sqrt{2n}}{\|\xi\|_{\MP^{Z},2}}~\int_{0}^{\infty}  t~\MP\big(\{\xi(Z_{1}) > t\big\}\big)~dt
\leq 
\frac{\sqrt{n}}{\sqrt{2}} \|\xi\|_{\MP^{Z},2}.
\end{align*}
Therefore
$$
\sup_{\theta\in\Theta}\big|\EEE\big[\phi_{n}\big(G(\theta,Z_{1})\big) - \EEE[G(\theta,Z_{1})]\big]\big|\leq \frac{\|\xi\|_{\MP^{Z},2}}{\sqrt{2 n}}\quad\mbox{for}~n\in\NNN, 
$$
and thus for arbitrary $n\in\NNN$
\begin{align*}
&
\MP\Big(\Big\{\sup_{\theta\in\overline{\Theta}}\big|\frac{1}{n}\sum_{j=1}^{n}G(\theta,Z_{j}) - \EEE[G(\theta,Z_{1})]\big|\geq \varepsilon\Big\}~\cap~B_{n}^{\xi}\Big)\\
&\leq 
\MP\Big(\Big\{\sup_{\theta\in\overline{\Theta}}\big|\sum_{j=1}^{n}\phi_{n}\big(G(\theta,Z_{j})\big) - n \EEE\big[\phi_{n}\big(G(\theta,Z_{1})\big)\big]\big|\geq 
n\varepsilon - \sqrt{n} \frac{\|\xi\|_{\MP^{Z},2}}{\sqrt{2}})
\Big\}\Big).
\end{align*}
We want to apply Theorem \ref{TalagrandBousquet} to the function class $\FFF_{n}$ consisting of all mappings $\phi_{n}\big(G(\theta,\cdot)\big) - \EEE\big[\phi_{n}\big(G(\theta,Z_{1})\big)\big]$ with $\theta\in\overline{\Theta}$, and we set 
$$
S_{n} := \sup_{\theta\in\overline{\Theta}}\big|\sum_{j=1}^{n}\phi_{n}\big(G(\theta,Z_{j})\big) - n \EEE\big[\phi_{n}\big(G(\theta,Z_{1})\big)\big]\big|.
$$
Combining (A 2) with \eqref{ersteEig} and property $\phi_{n}(0) = 0$, we have 
\begin{align*}
&
\big\|\phi_{n}\big(G(\theta,\cdot)\big) - \phi_{n}\big(G(\vartheta,\cdot)\big) \big\|_{\MQ,2} 
\leq 
\big\|G(\theta,\cdot) - G(\vartheta,\cdot) \big\|_{\MQ,2} \quad\mbox{for}~\theta, \vartheta\in\Theta,~\MQ\in \cM_{\textrm{\tiny fin}},\\
& 
\big|\phi_{n}\big(G(\theta,z)\big)|\leq \xi(z)\quad\mbox{for}~\theta\in\Theta,~z\in\RRR^{d}.
\end{align*}
In particular $\xi$ is not only a positive upper envelope of $\FFF^{\Theta}$ but also 
of the function classes $\FFF^{\overline{\Theta}} := \{G(\theta,\cdot)\mid \theta\in\overline{\Theta}\}$ and 
$\overline{\FFF}_{n} := \big\{\phi_{n}\big(G(\theta,\cdot)\big)\mid \theta\in\overline{\Theta}\big\}$, and 
$$
N\big(\eta \|\xi\|_{\MQ,2},\overline{\FFF}_{n},L^{2}(\MQ)\big)\leq 
N\big(\eta \|\xi\|_{\MQ,2},\FFF^{\overline{\Theta}},L^{2}(\MQ)\big)
\leq 
N\big(\eta \|\xi\|_{\MQ,2}/2,\FFF^{\Theta},L^{2}(\MQ)\big)
$$
holds for $\eta > 0$ and $\MQ\in \cM_{\textrm{\tiny fin}}$. So in view of \eqref{estimation with integrals} we obtain
\begin{equation}
\label{thirdupperbound}
\EEE[S_{n}]\leq \sqrt{n}~ 32 ~\sqrt{2}~\|\xi\|_{\MP^{Z},2}~ J(\FFF^{\Theta},\xi,1/4)
\end{equation}
%Putting together \eqref{firstupperbound} and \eqref{ersteEig}
Since $\xi$ is an envelope of $\overline{\FFF}_{n}$, we also have
\begin{equation}
\label{fourthupperbound}
\sup_{\theta\in\overline{\Theta}}\big|\phi_{n}\big(G(\theta,z)\big) - \EEE\big[\phi_{n}\big(G(\theta,Z_{1})\big)\big]\big|
\leq 
u_{n} := (\sqrt{2n} + 1)~\|\xi\|_{\MP^{Z},2} 
\end{equation}
for $n\in\NNN, z\in\RRR^{d}$. Finally, setting $\sigma^{2} := \EEE\big[\xi(Z_{1})^{2}\big]$,
%Finally, by \eqref{ersteEig}
\begin{equation}
\label{fifthupperbound}
%\sup_{\theta\in\overline{\Theta}}
\EEE\Big[\big|\phi_{n}\big(G(\theta,z)\big) - \EEE\big[\phi_{n}\big(G(\theta,Z_{1})\big)\big]\big|^{2}\Big]
\leq 
%\sup_{\theta\in\overline{\Theta}}
\EEE\big[\phi_{n}\big(G(\theta,Z_{1})\big)^{2}\big] \leq \sigma^{2}%~\mbox{for}~\theta\in\overline{\Theta}, n\in\NNN.
\end{equation}
for $\theta\in\overline{\Theta}$ and $n\in\NNN$.
\smallskip

Fix any $t > 0$, and let $n\in\NNN$ with $\varepsilon > \eta_{t,n}$ as well as 
$n\geq \|\xi\|_{\MP^{Z},2}^{2}/2$, where $\eta_{t,n}$ is as in the display of Theorem \ref{exponential bound}. Then $\sigma^{2}\leq u_{n}$, and with the help of \eqref{thirdupperbound} 
$$
n\varepsilon - \sqrt{n} \frac{\|\xi\|_{\MP^{Z},2}}{\sqrt{2}} = 
\frac{t}{t+1}~\big(n \varepsilon - \sqrt{n}\frac{\|\xi\|_{\MP^{Z},2}}{\sqrt{2}}\big)
+ \frac{n\varepsilon - \sqrt{n}\|\xi\|_{\MP^{Z},2}/\sqrt{2}}{t + 1}
\geq 
\frac{t n \varepsilon}{ 4 (t + 1)} + \EEE[S_{n}].
$$  
This implies 
\begin{align}
&\nonumber
\MP\Big(\Big\{\sup_{\theta\in\overline{\Theta}}\big|\frac{1}{n}\sum_{j=1}^{n}G(\theta,Z_{j}) - \EEE[G(\theta,Z_{1})]\big|\geq \varepsilon\Big\}~\cap~B_{n}^{\xi}\Big)\\
&\label{secondupperestimate}\leq 
\MP\Big(\Big\{\sup_{\theta\in\overline{\Theta}}\big|\sum_{j=1}^{n}\phi_{n}\big(G(\theta,Z_{j})\big) - n~\EEE\big[\phi_{n}\big(G(\theta,Z_{1})\big)\big]\big|\geq \frac{t n \varepsilon}{ 4 (t + 1)} + \EEE[S_{n}]\Big\}\Big).
\end{align}
Now, we are in the position to apply Theorem \ref{TalagrandBousquet} to $\FFF_{n}$ due to 
\eqref{thirdupperbound} - \eqref{fifthupperbound}, concluding
\begin{align*}
&
\MP\Big(\Big\{\sup_{\theta\in\overline{\Theta}}\big|\frac{1}{n}\sum_{j=1}^{n}\phi_{n}\big(G(\theta,Z_{j})\big) - \EEE\big[\phi_{n}\big(G(\theta,Z_{1})\big)\big]\big|\geq \frac{t n \varepsilon}{ 4 (t + 1)} + \EEE[S_{n}]\Big\}\Big)\\
&\leq 
\exp\left(- \frac{3 ~t~ n \varepsilon}{16~u_{n} (t+1)}~\ln\Big(1 + \frac{t~n~u_{n}~\varepsilon}{6 (t + 1)~(2 u_{n}~\EEE[S_{n}] + n~\sigma^{2})}\Big)\right)\\
&\leq 
\exp\left(\frac{-3 t^{2} n^{2}\varepsilon^{2}}{8 (t + 1)^{2}[24 u_{n} \EEE[S_{n}] + 12 n \sigma^{2} + t u_{n} n \varepsilon/(t + 1)]}\right).
\end{align*}
Furthermore $\sigma^{2} = \|\xi\|_{\MP^{Z},2}^{2} < \sqrt{n}\varepsilon\|\xi\|_{\MP^{Z},2}$, 
and $\EEE[S_{n}] < n\varepsilon/(t + 1)$ by \eqref{thirdupperbound}. Then the statement of Theorem \ref{exponential bound} may be derived easily from \eqref{application Talagrand 2}
along with \eqref{secondupperestimate}.
\hfill$\Box$
%%%%%%%%%%%%%%%%%%%%%%%%%%%%%%%%%%%%%%%%%%%%%%%%%%%%%%%%%%%%%%%%%%%%%%%%%%%%%%%%%%%%%%%%%%%%%%%%%%%%%%%%%%%%%%%%%%%%%%%%%%%%%%%%%%%%%%%%%%%%%%%%%%%%%%%%%%%%%%%%%%%%%%%%%%%%%%%%%%%%%%
\subsection{Proof of Proposition \ref{Hoelder-Bedingung auxiliary}}
\label{Beweis Hoelder Hilfsresultat}
Condition (H) allows to verify (A 3) for any at most countable dense subset $\overline{\Theta}$ of the compact set $\Theta$. 
\par 
Let $\overline{\theta}\in\Theta$ with $G(\overline{\theta},\cdot)$ being square $\MP^{Z}$-integrable. Then for any $\theta\in\Theta$ assumption (H) implies
$$
|G(\theta,z)|\leq |G(\overline{\theta},z)| + C(z)~d_{m,2}(\theta,\overline{\theta})^{\beta}\quad (z\in\RRR^{d}).
$$ 
In particular, $\xi := C(\cdot)~\Delta(\Theta)^{\beta} + |G(\overline{\theta},\cdot)|$ is square $\MP^{Z}$-integrable and satisfies (A 2). Hence it remains to show the inequalities for the terms $J(\FFF^{\Theta},\xi,\delta)$.
\medskip

For a totally bounded metric $d$ on $\Theta$ we shall use the symbol $N\big(\eta,\Theta,d\big)$ to denote the minimal number to cover $\Theta$ by closed $d$-balls with radius $\eta > 0$ and centers in $\Theta$.
\par
It may be verified easily that the restriction $d_{m,2}^{\beta}$ to $\Theta$ defines a totally bounded and complete metric on $\Theta$. By (H) we may observe 
$$
\|G(\theta,\cdot) - G(\vartheta,\cdot)\|_{\MQ,2}\leq \|C\|_{\MQ,2}~ d_{m,2}(\theta,\vartheta)^{\beta}\quad\mbox{for}~\MQ\in \cM_{\textrm{\tiny fin}},~\mbox{and}~\theta,\vartheta\in\Theta.
$$ 
Hence we obtain 
$$
N\big(\|\xi\|_{\MQ,2}~\eta,\FFF^{\Theta},L^{2}(\MQ)\big)\leq 
N\big(\Delta(\Theta)^{\beta} \eta,\Theta,d_{m,2}^{\beta}\big)\quad\mbox{for all}~\MQ\in \cM_{\textrm{\tiny fin}},~\eta > 0.
$$
Moreover, we have $\Theta\subseteq \{\gamma\in\RRR^{m}\mid d_{m,2}(\gamma,\overline{\theta})\leq \Delta(\Theta)\}$. 
Then 
we obtain from Lemma 2.5 in \cite{vandeGeer2000} that for every $\eta > 0$
\begin{eqnarray*}
N\big(\Delta(\Theta)^{\beta}~ \eta,\Theta,d_{m,2}^{\beta}\big)
\leq 
N\big(\Delta(\Theta)~ \eta^{1/\beta},\Theta,d_{m,2}\big)
\leq 
(8 + \eta^{1/\beta})^{m}/\eta^{m/\beta}.
\end{eqnarray*}
This implies for any $\delta\in ]0,1/2]$, using change of variable formula
\begin{align*}
J(\FFF^{\Theta},\xi,\delta)
&\leq 
\int_{0}^{\delta}\sqrt{\frac{m}{\beta}~\ln\big(2^{\beta/m}[8 + \delta^{1/\beta}]^{\beta}/\eta\big)}~d\eta\\
%&\leq 
%\int_{0}^{\delta}\sqrt{\frac{m}{\beta}~\ln\big(2^{[(m + 1)\beta + m]/m}/\eta\big)}~d\eta\\
&\leq
\delta \int_{0}^{1}\sqrt{\frac{m}{\beta}~\ln\Big(\frac{2^{[(3 m + 1)\beta + m]/m}/\delta}{\eta}\Big)}~d\eta.
\end{align*}
Now, we may invoke \eqref{Integralabschaetzung} with $v := m/\beta$ and $K := 2^{[(3 m + 1)\beta + m]/m}/\delta$ 
%($\geq e$ for $\delta\in ]0,1/2]$) 
to derive the remaining part of 
Proposition \ref{Hoelder-Bedingung auxiliary}.
\hfill$\Box$

%%%%%%%%%%%%%%%%%%%%%%%%%%%%%%%%%%%%%%%%%%%%%%%%%%%%%%%%%%%%%%%%%%%%%%%%%%%%%%%%%%%%%%%%%%%%%%%%%%%%%%%%%%%%%%%%%%%%%%%%%%%%%%%%%%%%%%%%%%%%%%%%%%%%%%%%%%%%%%%%%%%%%%%%%%%%%%%%%%%%%%%%
\subsection{Proof of Proposition \ref{startingpoint}}
\label{proof of Proposition starting point}
We start the proof of Proposition \ref{startingpoint} 
with the following observation induced by representation (PH).
\begin{equation}
\label{startingoberservation}
G(\theta,z)  
= 
\sum_{i = 1}^{r}f^{i}(\theta,z)~G^{i}(\theta,z)\quad\mbox{for}~\theta\in\Theta, z\in\RRR^{d}.
\end{equation}
The mappings $\Lambda_{il}(\theta,\cdot), G^{i}(\theta,\cdot)$ are Borel measurable by assumption, in particular Borel measurability of $f^{i}(\theta,\cdot)$ holds. Hence by \eqref{startingoberservation} the assumption (A 1) is fulfilled. Moreover, let the mappings $\xi_{1},\ldots,\xi_{r},\xi$ be defined as in the display of Proposition \ref{startingpoint} with square $\MP^{Z}$-integrable $G^{1}(\overline{\theta},\cdot),\ldots,G^{r}(\overline{\theta},\cdot)$ for some $\overline{\theta}\in\Theta$. Then by construction, the mapping $\xi$ is also square $\MP^{Z}$-integrable because the mappings $\xi_{1},\ldots,\xi_{r}$ are assumed to be bounded. In particular it satisfies (A 2) by \eqref{startingoberservation} again. Therefore it remains to verify the claimed upper estimates of the terms $J(\FFF^{\Theta},\xi,\delta)$.
\medskip

We need some further preparation from the theory of empirical process theory. 
%ForTypical examples for function classes satisfying conditions (\ref{random entropies bounded}) or (\ref{random entropies unbounded}) are provided by the so-called VC-subgraph classes. 
To recall, define for a collection $\cB$ of subsets of $\RRR^{d}$, and $z_{1},\dots,z_{n}\in\RRR^{d}$
$$
\Delta_{n}(\cB,z_{1},\dots,z_{n})~:=~\mbox{cardinality of}~\left\{B~\cap~\{z_{1},\dots,z_{n}\}\mid B\in\cB\right\}.
$$
Then 
$$
V(\cB)~:=~\inf~\Big\{n\in\NNN\mid \max_{z_{1},\dots,z_{n}\in\RRR^{d}}\Delta_{n}(\cB,z_{1},\dots,z_{n}) < 2^{n}\Big\}\quad(\inf\emptyset~:=~\infty)
$$
is known as the \textit{index} of $\cB$ (see \cite{vanderVaartWellner1996}, p. 135). In case of finite index, $\cB$ is known as a so called \textit{VC-class} (see \cite{vanderVaartWellner1996}, p. 135). The concept of VC-classes may be carried over from sets to functions in the following way. A set $\FFF$ of Borel measurable real valued functions on $\RRR^{d}$ is defined to be a \textit{VC-subgraph class} or a \textit{VC-class} if 
the corresponding collection $\big\{\{(z,t)\in\RRR^{d}\times\RRR\mid h(z) > t\}~\mid~h\in\FFF\big\}$ of subgraphs is a {\textit VC-class} (\cite{vanderVaartWellner1996}, p. 141). Its {\textit VC-index} $V(\FFF)$ coincides with the index of the subgraphs. The significance of VC-subgraph classes stems from the fact 
%that for every VC-subgraph class $\FFF$ 
that there there exists some  universal constant $K_{\textrm{\tiny VC}} \geq 1$ such that for every VC-subgraph class $\FFF$ and any $\MP^{Z}$-integrable positive envelope $C_{\FFF}$ of $\FFF$
\begin{equation*}
%\label{VC-subgraph covering numbers}
\sup_{\MQ\in\cM_{\textrm{\tiny fin}}}N\big(\varepsilon \|C_{\FFF}\|_{\MQ,2},\FFF,L^{2}(\MQ)\big)\leq K_{\textrm{\tiny VC}}~V(\FFF)~(16 e)^{V(\FFF)}\big(1/\varepsilon\big)^{2 [V(\FFF) - 1]}\quad\mbox{for}~\varepsilon\in ]0,1[
\end{equation*}
(see \cite[Theorem 9.3]{Kosorok2008} or \cite[Theorem 2.6.7]{vanderVaartWellner1996}). 
\par
For our purposes we are interested in more explicit upper estimations of the covering numbers. This may be achieved upon Corollary 3 in \cite{Haussler1995} which we recall now for the convenience of the reader. 
\begin{proposition}
\label{Haussler}
Let $\FFF = \{\eins_{B}\mid B\in\cB\}$, where $\cB$ denotes some VC-class. Then 
$$
\sup_{\MQ\in\cM_{\textrm{\tiny fin}}}N\big(\varepsilon,\FFF,L^{1}(\MQ)\big)\leq e~V(\FFF)~\big(2 e/\varepsilon\big)^{V(\FFF) - 1}\quad\mbox{for}~\varepsilon\in ]0,1[.
$$
\end{proposition}
Once we have upper estimates for covering numbers of VC-classes w.r.t. the $L^{1}$-norms, it is well-known from the theory of empirical process theory how to derive upper estimates for covering numbers of VC-subgraph classes w.r.t. the $L^{2}$-norm. We obtain the following result.
\begin{corollary}
\label{VC-subgraph covering numbers}
Let $\FFF$ be any VC-subgraph class with some arbitrary positive envelope $C_{\FFF}$. Then the inequality
\begin{equation*}
\sup_{\MQ\in\cM_{\textrm{\tiny fin}}}N\big(\varepsilon \|C_{\FFF}\|_{\MQ,2},\FFF,L^{2}(\MQ)\big)\leq e~V(\FFF)~\big(4 e^{1/2}/\varepsilon\big)^{2 [V(\FFF) - 1]}\quad\mbox{for}~\varepsilon\in ]0,1[
\end{equation*}
holds. 
\end{corollary}
\begin{proof}
The proof mimicks the proof of Theorem 9.3 in \cite{Kosorok2008} or the proof of Theorem 2.6.7 in \cite{vanderVaartWellner1996}.
\par
Let $\FFF_{\cB} := \{\eins_{B}\mid B\in\cB\}$, where $\cB$ denotes the collection of subgraphs corresponding to $\FFF$. In the first step one may obtain
\begin{equation}
\label{firstStep}
N\big(\varepsilon \|C_{\FFF}\|_{\MQ,1},\FFF,L^{1}(\MQ)\big)\leq N\big(\varepsilon/2,\FFF_{\cB},L^{1}(\MQ)\big)\quad\mbox{for}~\MQ\in\cM_{\textrm{\tiny fin}},~\varepsilon\in ]0,1[.
\end{equation}
In the second step any $\MQ\in\cM_{\textrm{\tiny fin}}$ is associated with the probability measure $\MQ_{C_{\FFF}}\in\cM_{\textrm{\tiny fin}}$, defined by 
$\MQ_{C_{\FFF}}(B) := \EEE_{\MQ}[\eins_{B}C_{\FFF}]/\EEE_{\MQ}[C_{\FFF}]$. Then it can be shown that 
\begin{equation}
\label{secondStep}
N\big(\varepsilon \|C_{\FFF}\|_{\MQ,2},\FFF,L^{2}(\MQ)\big)\leq 
N\big(\varepsilon^{2} \|C_{\FFF}\|_{\MQ_{C_{\FFF}},1}/4,\FFF,L^{1}(\MQ_{C_{\FFF}})\big)
\end{equation}
holds for $\varepsilon\in ]0,1[$. Then, combining \eqref{firstStep} and \eqref{secondStep} with Haussler's result Proposition \ref{Haussler}, we may complete the proof.
\end{proof}
%This inequality may be derived from Corollary 3 in \cite{Haussler1995} in exactly the same way as Theorem 9.3 in \cite{Kosorok2008} has been show upon Theorem 9.2 there.(see \cite[Theorem 2.6.7]{vanderVaartWellner1996}). {\color{red}vielleicht Herleitung detaillierter?}
%\medskip
In view of (\ref{startingoberservation}) the following auxiliary results reveals that the classes $\FFF^{\Theta}$ is built upon specific VC-subgraph classes. This will be crucial for deriving the result of Proposition \ref{startingpoint}. 
%Remember the notion of VC-subgraph classes introduced in Section 6.%\ref{error bounds}. 
\begin{lemma}
\label{starting VC-classes I}
For every $i\in\{1,\ldots,r\}$ and any nonvoid $\overline{\Theta}\subseteq\Theta$, the 
set $\FFF_{i,\overline{\Theta}}$ consisting of all $f^{i}(\theta,\cdot)$ with 
$\theta\in\overline{\Theta}$ is a VC-subgraph class with index $V(\FFF_{i,\overline{\Theta}})\leq (m + 2)~ s_{i} + 1$.
\end{lemma}
\begin{proof}
Let $\overline{\Theta}\subseteq\Theta$ nonvoid, and let $i\in\{1,\ldots,r\}$. 
Set $\overline{I}_{il} := \{t - a^{i}_{l}\mid t\in I_{il}\}$, and $B_{il}(\theta) := \Lambda_{il}(\theta,\cdot)^{-1}(\overline{I}_{il})$ for $i\in\{1,\ldots,r\}$, $l\in\{1,\ldots,s_{i}\}$, $\theta\in\Theta$. Furthermore, $\{e_{1},\ldots,e_{m}\}$ denotes the standard basis of $\RRR^{m}$.
\par
The linear hull of $\{\Lambda_{il}(\theta,\cdot)\mid\theta\in\overline{\Theta}\}$ is spanned by $\big\{\Lambda_{il}(y,\cdot)\mid y\in\{e_{1},\ldots,e_{m},0\}\big\}$ so that its dimension does not exceed $m + 1$. Thus by Lemma 2.6.15 in \cite{vanderVaartWellner1996} the set $\{\Lambda_{il}(\theta,\cdot)\mid\theta\in\overline{\Theta}\}$ is a VC-subgraph class with index $\leq m + 3$. Moreover, $\eins_{\overline{I}_{il}}$ is a monotone function, and hence $\{\eins_{B_{il}(\theta)}\mid\theta\in\overline{\Theta}\} = \{\eins_{\overline{I}_{il}}\circ\Lambda_{il}(\theta,\cdot)\mid\theta\in\overline{\Theta}\}$ is a VC-subgraph class with index $\leq m + 3$ (see \cite[Lemma 9.9, (viii)]{Kosorok2008}). 
Finally, $\FFF_{i,\overline{\Theta}}$ is a subset of all functions $\min_{l=1,\ldots,s_{i}}\eins_{B_{il}(\theta)}$ with $\theta\in\overline{\Theta}$ which implies that $\FFF_{i,\overline{\Theta}}$ is a VC-subgraph class with index
$$
V(\FFF_{i,\overline{\Theta}})\leq\sum_{l=1}^{s_{i}}V\big(\{\eins_{B_{il}(\theta)}\mid\theta\in\overline{\Theta}\}\big) - (s_{i} - 1) \leq (m + 2)~s_{i} + 1\quad i\in\{1,\ldots,r\}
$$
(see \cite[Lemma 9.9, (i)]{Kosorok2008}). This completes the proof.
\end{proof}
By (PH) each $G^{i}$ satisfies condition (H). Therefore, as an immediate consequence of the Proposition \ref{Hoelder-Bedingung auxiliary} we obtain the next auxiliary result.
\begin{lemma}
\label{starting VC-classes II}
Define for nonvoid $\overline{\Theta}\subseteq \Theta$ and $i\in\{1,\ldots,r\}$ the set 
$\overline{\FFF}_{i,\overline{\Theta}}$ of all $G^{i}(\theta,\cdot)$ with  $\theta\in\overline{\Theta}$. Furthermore let $\Delta(\overline{\Theta})$ denote the diameter of $\overline{\Theta}$ w.r.t. the Euclidean metric, and let $\beta_{i}\in ]0,1]$ and $C_{i}:\RRR^{d}\rightarrow\RRR$ be as in (PH). If $\Delta(\overline{\Theta}) > 0$, and if $G^{i}(\overline{\theta},\cdot)$ is square $\MP^{Z}$-integrable for some $\overline{\theta}\in\overline{\Theta}$, then 
$\overline{\xi}_{i} := \Delta(\overline{\Theta})^{\beta_{i}}~C_{i}(\cdot) + |G^{i}(\overline{\theta},\cdot)|$ is a square $\MP^{Z}$-integrable positive envelope of $\overline{\FFF}_{i,\overline{\Theta}}$, and
$$
\sup_{\MQ\in \cM_{\textrm{\tiny fin}}}N\big(\varepsilon \|\overline{\xi}_{i}\|_{\MQ,2},\overline{\FFF}_{i,\overline{\Theta}}, L^{2}(\MQ)\big)\leq 9^{m}~\varepsilon^{-m/\beta_{i}}\quad\mbox{for}~\varepsilon \in ]0,1[.
$$ 
\end{lemma}

Now, we are ready to finish the proof Proposition \ref{startingpoint}.
\medskip

We consider the function class $\FFF_{i}$ consisting of all mappings $f^{i}(\theta,\cdot)\cdot G^{i}(\theta,\cdot)$
with $\theta\in\Theta$ for $i\in\{1,\ldots,r\}$. The significance of these function classes for our purposes stems from representation (\ref{startingoberservation}). Note that $\hat{\xi}_{i} := \xi_{i}\cdot\overline{\xi}_{i}$ defines a positive envelope of $\FFF_{i}$ for $i\in\{1,\ldots,r\}$, where $\overline{\xi}_{i} := \Delta(\Theta)^{\beta_{i}}~C_{i}(\cdot) + |G^{i}(\overline{\theta},\cdot)|$ (see Lemma \ref{starting VC-classes II}). Our aim is to find explicit upper estimates of the 
covering number $N\big(\varepsilon \|\hat{\xi}_{i}\|_{\MQ,2},\FFF_{i},L^{2}(\MQ)\big)$ with 
$\MQ\in\cM_{\textrm{\tiny fin}}$.
\par
Fix $i\in\{1,\ldots,r\}$. First of all, $\FFF_{\textrm{\tiny PH}}^{i}$ is a VC-subgraph class with index $V(\FFF_{\textrm{\tiny PH}}^{i})\leq (m + 2)~s_{i} + 1$ by Lemma \ref{starting VC-classes I}. 
%and $\overline{\FFF}_{\textrm{\tiny PL}}^{i}$ is a VC-subgraph class with index $V(\overline{\FFF}_{\textrm{\tiny PL}}^{i})\leq 4$ due to  Lemma \ref{starting VC-classes II}. Furthermore $\overline{\xi}_{i} :=  |\Lambda_{i}| + \eta_{i}^{G}$ is a positive envelope of $\overline{\FFF}_{\textrm{\tiny PL}}^{i}$. 
Then we may conclude from Corollary \ref{VC-subgraph covering numbers} 
\begin{align*}
&\nonumber
\sup_{\MQ\in\cM_{\textrm{\tiny fin}}}N\big(\varepsilon \|\xi_{i}\|_{\MQ,2},\FFF_{\textrm{\tiny PH}}^{i},L^{2}(\MQ)\big)
\leq 
e ([m + 2]~s_{i} + 1) \big(4 e^{1/2}/\varepsilon\big)^{2 (m + 2)~s_{i}}\quad\mbox{for}~\varepsilon\in ]0,1[.
%\\
%& 
%\sup_{\MQ\in\cM_{\textrm{\tiny fin}}}N\big(\varepsilon \|\overline{\xi}_{i}\|_{\MQ,2},\overline{\FFF}_{\textrm{\tiny PL}}^{i},L^{2}(\MQ)\big)
%\leq 4 e \big(4 e^{1/2}/\varepsilon\big)^{6} \quad\mbox{for}~\varepsilon\in ]0,1[.
\end{align*}
Moreover, we have 
\begin{align*}
&\nonumber
\sup_{\MQ\in\cM_{\textrm{\tiny fin}}}N\big(\varepsilon \|\hat{\xi}_{i}\|_{\MQ,2},\FFF_{i},L^{2}(\MQ)\big)\\ 
&
\leq 
\sup_{\MQ\in\cM_{\textrm{\tiny fin}}}N\big(\varepsilon \|\xi_{i}\|_{\MQ,2}/4,\FFF_{\textrm{\tiny PH}}^{i},L^{2}(\MQ)\big)\cdot \sup_{\MQ\in\cM_{\textrm{\tiny fin}}}N\big(\varepsilon \|\overline{\xi}_{i}\|_{\MQ,2}/4,\overline{\FFF}_{\textrm{\tiny PH}}^{i},L^{2}(\MQ)\big)
%\quad\mbox{for}~\varepsilon\in ]0,1[.
\end{align*}
for $\varepsilon\in ]0,1[$ 
(see Corollary A.1. in supplement to 
\cite{ChernozhukovEtAl2014} or proof of Theorem 9.15 in \cite{Kosorok2008}).
Hence in view of Lemma \ref{starting VC-classes II} we end up with. 
\begin{align}
&\nonumber
\sup_{\MQ\in\cM_{\textrm{\tiny fin}}}N\big(\varepsilon \|\hat{\xi}_{i}\|_{\MQ,2},\FFF_{i},L^{2}(\MQ)\big)
\\ 
&\label{Abschaetzung3}
%\\ 
%&\label{Abschaetzung3}
\leq
9^{m}~16^{(m+2) s_{i}}~e^{1 + (m + 2)~s_{i}}~([m + 2]~s_{i} + 1)~\big(4/\varepsilon\big)^{2 (m + 2) s_{i} + m/\beta_{i}}
%4~e^{2}~(s_{i} + 1)~\big(16 e^{1/2}/\varepsilon\big)^{2 [s_{i} + 3]}
%\quad\mbox{for}~i\in\{1,\ldots,r\}, \varepsilon\in ]0,1[.
\end{align}
for $i\in\{1,\ldots,r\}, \varepsilon\in ]0,1[$. Next, fix $\MQ\in\cM_{\textrm{\tiny fin}}$, $\varepsilon > 0$. Let $h^{i},\overline{h}^{i}\in\FFF_{i}$ such that the inequality $\|h^{i}-\overline{h}^{i}\|_{\MQ,2}\leq \varepsilon \|\hat{\xi}_{i}\|_{\MQ,2}/r$ holds for $i =1,\ldots,r$. Then 
by inequality $\sqrt{\sum_{i=1}^{r}t_{i}}\geq\sum_{i=1}^{r}\sqrt{t_{i}}/r$ for $t_{1},\ldots,t_{r}\geq 0$
\begin{eqnarray*}
\|\sum_{i=1}^{r} h^{i} - \sum_{i=1}^{r}\overline{h}^{i}\|_{\MQ,2}
\leq 
\sum_{i=1}^{r}\|h^{i}-\overline{h}^{i}\|_{\MQ,2} 
\leq 
\frac{\varepsilon}{r}\sum_{i=1}^{r}\|\hat{\xi}_{i}\|_{\MQ,2}
\leq 
\varepsilon \|\sum_{i=1}^{r}\hat{\xi}_{i}\|_{\MQ,2}.
\end{eqnarray*}
Thus by construction of $\xi$ along with (\ref{startingoberservation})
\begin{equation}
\label{random entropies Summen}
N\big(\varepsilon \|\xi\|_{\MQ,2},\FFF^{\Theta},L^{2}(\MQ)\big)\leq 
\prod_{i=1}^{r}N\big(\varepsilon \|\hat{\xi}_{i}\|_{\MQ,2}/r,\FFF_{i},L^{2}(\MQ)\big)
\end{equation}
for $\MQ\in\cM_{\textrm{\tiny fin}}$ and $\varepsilon > 0$. 
\medskip

Combining (\ref{Abschaetzung3}) and (\ref{random entropies Summen}), we obtain for $\delta\in ]0,1]$ by change of variable formula
$$
J(\FFF^{\Theta},\xi,\delta) = \delta \int_{0}^{1}\hspace*{-0.2cm}\sup_{\MQ\in\cM_{\textrm{\tiny fin}}}\sqrt{\ln\Big(2 N\big(\delta \varepsilon \|\xi\|_{\MQ,2},\FFF^{\Theta},L^{2}(\MQ)\big)\Big)}~d\varepsilon
\leq \delta \int_{0}^{1}\sqrt{v \ln(K_{\delta}/\varepsilon)}~d\varepsilon,
$$
where
$$
v := 2 (m + 2)\sum_{i=1}^{r}s_{i} + m \sum_{i=1}^{r}1/\beta_{i}
$$
and 
$$
K_{\delta} := \frac{4~r~\big(2\cdot 9^{r\cdot m} 16^{(m + 2)~\sum_{i=1}^{r} s_{i}} e^{r + (m + 2)~\sum_{i=1}^{r} s_{i}} ~\prod_{i=1}^{r}([m + 2]~s_{i} + 1)\big)^{1/v}}{\delta}.
$$
Now, we may finish the proof of Proposition \ref{startingpoint} via 
\eqref{Integralabschaetzung} by routine calculations.
\hfill$\Box$

%%%%%%%%%%%%%%%%%%%%%%%%%%%%%%%%%%%%%%%%%%%%%%%%%%%%%%%%%%%%%%%%%%%%%%%%%%%%%%%%%
%%%%%%%%%%%%%%%%%%%%%%%%%%%%%%%%%%%%%%%%%%%%%%%%%%%%%%%%%%%%%%%%%%%%%%%%%%%%%%%%%
\subsection{Proofs of the results from Section \ref{upper semideviations}}
\label{Beweise upper semideviations}
As a first result we shall show Lemma \ref{Rueckfuehrung}.
\smallskip

\noindent 
\textbf{Proof of Lemma \ref{Rueckfuehrung}:}\\[0.1cm]
Let $n\in\NNN$ and $a\in ]0,1]$. By choice of the random variable $\xi$ we may observe 
\begin{align*}
\inf_{\theta\in\Theta}\cR_{\rho_{p,a}}\big(F_{\theta}\big)\geq -\EEE[\xi] > - \infty\quad\mbox{and}\quad
\inf_{\theta\in\Theta}\cR_{\rho_{p,a}}\big(\hat{F}_{n,\theta}\big)\geq - \frac{1}{n}\sum_{j=1}^{n}\xi(Z_{j}) > -\infty.
\end{align*}
Moreover, using Minkowski's inequality, by representations \eqref{optimization general III} and \eqref{SAA upper semideviations} we have for nonvoid $\overline{\Theta}\subseteq\Theta$
\begin{align*}
&
\big|\inf_{\theta\in\overline{\Theta}}\cR_{\rho_{p,a}}\big(\hat{F}_{n,\theta}\big)-\inf_{\theta\in\overline{\Theta}}\cR_{\rho_{p,a}}\big(F_{\theta}\big)\big|
\\
&\leq 
(1 + a) \sup_{\theta\in\overline{\Theta}}\big|\frac{1}{n}\sum_{j=1}^{n}G(\theta,Z_{j}) - \EEE[G(\theta,Z_{1})]\big|\\
&\quad 
+ 
a\sup_{\theta\in\overline{\Theta}}\Big|\Big(\frac{1}{n}\sum_{j=1}^{n}G_{p}(\theta,Z_{j})\Big)^{1/p} - \Big(\EEE[G_{p}(\theta,Z_{1})]\Big)^{1/p}\Big|.\\
\end{align*}
Since $|t^{1/p}- s^{1/p}|\leq |t-s|^{1/p}$ holds for $t,s\geq 0$, we end up with  
\begin{align*}
\big|\inf_{\theta\in\overline{\Theta}}\cR_{\rho_{p,a}}\big(\hat{F}_{n,\theta}\big)-\inf_{\theta\in\overline{\Theta}}\cR_{\rho_{p,a}}\big(F_{\theta}\big)\big|
&\leq 
(1 + a) \sup_{\theta\in\overline{\Theta}}\big|\frac{1}{n}\sum_{j=1}^{n}G(\theta,Z_{j}) - \EEE[G(\theta,Z_{1})]\big|\\
&\qquad 
+ a\sup_{\theta\in\overline{\Theta}}\big|\frac{1}{n}\sum_{j=1}^{n}G_{p}(\theta,Z_{j}) - \EEE[G_{p}(\theta,Z_{1})]\big|^{1/p}.
\end{align*}
Now, the proof may be finished easily.
\hfill$\Box$
\medskip

\noindent
\textbf{Proof of Lemma \ref{ReduktionCountable}:}\\[0.1cm]
Let $\overline{\Theta}\subseteq\Theta$ from (A 3'). For $\theta\in\Theta$ we may select by (A 3') a sequence $(\vartheta_{k})_{k\in\NNN}$ in $\overline{\Theta}$ such 
that $\EEE[|G(\vartheta_{k},Z_{1}) - G(\theta,Z_{1})|]\to 0$, and thus $G(\vartheta_{k},Z_{1})\to G(\theta,Z_{1})$ in probability by application of Markov's inequality. This implies $G_{p}(\vartheta_{k},Z_{1})\to G_{p}(\theta,Z_{1})$ in probability. Furthermore we have upper estimation $|G_{p}(\vartheta_{k},Z_{1})|\leq \big(\xi(Z_{1}) + \EEE[\xi(Z_{1})]\big)^{p}$ for $k\in\NNN$, and $\xi$ is integrable of order $p$ by assumption. Thus the application of Vitalis' theorem (see \cite[Proposition 21.4]{Bauer2001}) yields $\EEE[G_{p}(\vartheta_{k},Z_{1})]\to \EEE[G_{p}(\theta,Z_{1})]$. Thus we have shown for any $\theta\in\Theta$
\begin{equation}
\label{Doppelkonvergenz}
\inf_{\vartheta\in\overline{\Theta}}\Big\{\big|\EEE[G(\vartheta,Z_{1})] - \EEE[G(\theta,Z_{1})]\big| + \big|\EEE[G_{p}(\vartheta,Z_{1})] - \EEE[G_{p}(\theta,Z_{1})]\big|\Big\} = 0.
\end{equation}
In view of representation \eqref{optimization general III}, statement 1) follows immediately from \eqref{Doppelkonvergenz}.
\medskip

Next, fix $n\in\NNN$, choose the $\big(\MP^{Z}\big)^{n}$-null set $N_{n}$ according to (A 3'), and consider any vector $(z_{1},\ldots,z_{n})\in\RRR^{dn}\setminus N_{n}$. For $\theta\in\Theta$ we may find via (A 3') some sequence $(\vartheta_{k})_{k\in\Theta}$ in $\overline{\Theta}$ such that 
$\EEE[G(\vartheta_{k},Z_{1})]\to \EEE[G(\theta,Z_{1})]$ and $G(\vartheta_{k},z_{j})\to G(\theta,z_{j})$ for $j\in\{1,\ldots,n\}$. Then $G_{p}(\vartheta_{k},z_{j})\to G_{p}(\theta,z_{j})$ for every $j\in\{1,\ldots,n\}$. In particular statement 2) may be concluded from \eqref{Doppelkonvergenz} along with (A 3').
\medskip

Let us define the set $A_{n} := \{(Z_{1},\ldots,Z_{n})\in \RRR^{d n}\setminus N_{n}\}\in\cF$. Note $\MP(A_{n}) = 1$. Fix $\omega\in\Omega$. By (A 3') there exists for any $\theta\in\Theta$ some sequence $(\vartheta_{k})_{k\in\Theta}$ in $\overline{\Theta}$ satisfying $G\big(\vartheta_{k},Z_{j}(\omega)\big)\to G\big(\theta,Z_{j}(\omega)\big)$ for $j\in\{1,\ldots,n\}$. Then, drawing on representation \eqref{SAA upper semideviations}, the convergence $\cR_{\rho_{p,a}}(\hat{F}_{n,\vartheta_{k}})(\omega)\to \cR_{\rho_{p,a}}(\hat{F}_{n,\theta})(\omega)$ may be verified easily for every $a\in ]0,1]$. This shows statement 3), recalling $\MP(A_{n}) = 1$. The proof is complete.

\hfill$\Box$
\medskip

\noindent
\textbf{Proof of Lemma \ref{auxiliary I}:}\\[0.1cm]
First of all
$$
G(\theta,z)  - \EEE[G(\theta,Z_{1})]\leq \xi + \big(\EEE[\xi(Z_{1})]\vee 1\big)
$$
holds for $\theta\in\Theta$ and $z\in\RRR^{d}$ so that $\xi_{p}$ is a positive envelope of $\FFF^{\Theta,p}$.
\par
Next, let $s,t,u\in [0,\infty[$ with $u\geq t\vee s$. The mapping $f: ]1,\infty[\rightarrow\RRR$, defined by $f(q) = |s^{q}- t^{q}|$ is nondecreasing. Hence 
$|s^{p} - t^{p}|\leq |s^{\lceil p\rceil} - t^{\lceil p\rceil}|$, using notation $\lceil p\rceil := \min [p,\infty[\cap\NNN$.  
\par
Moreover, $|s^{k+1}- t^{k+1}| = (s\vee t) |s^{k} - t^{k}| + |s - t| (s\wedge t)^{k}$ holds for $k\in\NNN$. Then it may be shown by induction that $|s^{k} - t^{k}|\leq |s - t| (2 u)^{k-1}$ is valid for every $k\in\NNN$. In particular, we end up with the inequality 
$|s^{p} - t^{p}|\leq |s - t|(2u)^{\lceil p\rceil - 1}$. As a further consequence we may observe for $\theta, \vartheta\in\Theta$ and $z\in\RRR^{d}$
\begin{align*}
&|G_{p}(\theta,z) - G_{p}(\vartheta,z)|^{2}\\
&\leq 
\big|\big(G(\theta,z) - \EEE[G(\theta,Z_{1})]\big)^{+} - \big(G(\vartheta,z) - \EEE[G(\vartheta,Z_{1})]\big)^{+}\big|^{2}\big(2\xi(z) + 2 \EEE[\xi(Z_{1})]\big)^{2\lceil p\rceil - 2}\\
&\leq 
2^{2(p + 1)}\xi_{p}(z)^{2 p/(p+1)}\Big(|G(\theta,z) - G(\vartheta,z) \big|^{2} + 
|\EEE[G(\theta,Z_{1})] - \EEE[G(\vartheta,Z_{1})]|^{2}\Big).
\end{align*}
The positive envelope $\xi_{p}$ of $\FFF^{\Theta,p}$ is square $\MP^{Z}$-square integrable by assumption, and the constant $\EEE[\xi(Z_{1})]$ may be viewed as an positive envelope of the class $I$ which gathers all constant functions $\EEE[G(\theta,Z_{1})]$ ($\theta\in\Theta)$. We may apply Theorem 2.10.20 from \cite{vanderVaartWellner1996} which leads to 
\begin{align*}
&
\int_{0}^{\delta}\sup_{\MQ\in \cM_{\textrm{\tiny fin}}}\sqrt{\ln\big(N\big(\varepsilon~2^{p+1}\|\xi_{p}^{p/(p+1)}\sqrt{\xi^{2} + \EEE[\xi(Z_{1})]^{2}}\|_{\MQ,2},\FFF^{\Theta,p},L^{2}(\MQ)\big)\big)}~d\varepsilon\\
&\leq 
\int_{0}^{\delta}\sup_{\MQ\in \cM_{\textrm{\tiny fin}}}\sqrt{\ln\big(N\big(\varepsilon~\|\xi\|_{\MQ,2}/2,\FFF^{\Theta},L^{2}(\MQ)\big)\big)}~d\varepsilon 
+ 
\int_{0}^{\delta}\sqrt{\ln\big(N\big(\varepsilon~\EEE[\xi(Z_{1})]/2,I\big)}~d\varepsilon\\
&\leq
%\int_{0}^{\delta}\sup_{\MQ\in \cM_{\textrm{\tiny fin}}}\sqrt{\ln\big(N\big(\varepsilon~\|\xi\|_{\MQ,2},\FFF^{\Theta},L^{2}(\MQ)\big)\big)}~d\varepsilon 
2~ J(\FFF^{\Theta},\xi,\delta/2)
+ 
\int_{0}^{\delta}\sqrt{\ln\big(N\big(\varepsilon~\EEE[\xi(Z_{1})]/4,\big[-\EEE[\xi(Z_{1})],\EEE[\xi(Z_{1})]\big]\big)}~d\varepsilon
\end{align*}
for $\delta > 0$, where for $J\in\big\{I, \big[-\EEE[\xi(Z_{1})],\EEE[\xi(Z_{1})]\big]\big\}$ and $\eta > 0$ we denote by the symbol $N\big(\eta~\EEE[\xi(Z_{1})] ,J\big)$ the minimal number to cover $J$ by closed intervals of the 
form $[x_{i} -\eta~\EEE[\xi(Z_{1})], x_{i} + \eta~\EEE[\xi(Z_{1})]]$ with $x_{i}\in J$. It is easy to check that the inequality $N\big(\varepsilon~\EEE[\xi(Z_{1})]/4,\big[-\EEE[\xi(Z_{1})],\EEE[\xi(Z_{1})]\big]\big)\leq 8/\varepsilon$ holds for $\varepsilon > 0$. Hence we may invoke the change of variable formula along with \eqref{Integralabschaetzung} which yields
\begin{align*}
&
\int_{0}^{\delta}\sup_{\MQ\in \cM_{\textrm{\tiny fin}}}\sqrt{\ln\big(N\big(\varepsilon~2^{p+1}\|\xi_{p}^{p/(p + 1)}\sqrt{\xi^{2} + \EEE[\xi(Z_{1})]^{2}}\|_{\MQ,2},\FFF^{\Theta,p},L^{2}(\MQ)\big)\big)}~d\varepsilon\\
&\leq 
2 ~J(\FFF^{\Theta},\xi,\delta/2) + \int_{0}^{\delta}\sqrt{\ln(8/\varepsilon)}~d\varepsilon
= 
2~ J(\FFF^{\Theta},\xi,\delta/2) + \delta \int_{0}^{1}\sqrt{\ln\big((8/\delta)/\varepsilon\big))}~d\varepsilon\\ 
&\leq 
2~J(\FFF^{\Theta},\xi,\delta/2) + 2 \delta \sqrt{\ln(8/\delta)}\quad\mbox{for}~\delta\in ]0,1[.
\end{align*} 
Since $\|\xi_{p}^{p/(p + 1)}\sqrt{\xi^{2} + \EEE[\xi(Z_{1})]^{2}}\|_{\MQ,2}\leq\|\xi_{p}\|_{\MQ,2}$ is valid for any $\MQ\in\cM_{\textrm{\tiny fin}}$, we may further conclude, using change of variable formula again,
\begin{align*}
&
\int_{0}^{\delta}\sup_{\MQ\in \cM_{\textrm{\tiny fin}}}\sqrt{\ln\big(N\big(\varepsilon~\|\xi_{p}\|_{\MQ,2},\FFF^{\Theta,p},L^{2}(\MQ)\big)\big)}~d\varepsilon\\
&\leq 
\int_{0}^{\delta}\sup_{\MQ\in \cM_{\textrm{\tiny fin}}}\sqrt{\ln\big(N\big(\varepsilon~\|\xi_{p}^{p/(p + 1)}\sqrt{\xi^{2} + \EEE[\xi(Z_{1})]^{2}}\|_{\MQ,2},\FFF^{\Theta,p},L^{2}(\MQ)\big)\big)}~d\varepsilon\\
&\leq 
2^{p+1}\int_{0}^{\delta/2^{p+1}}\hspace*{-0.5cm}\sup_{\MQ\in \cM_{\textrm{\tiny fin}}}\sqrt{\ln\big(N\big(\eta~2^{p+1}\|\xi_{p}^{p/(p + 1)}\sqrt{\xi^{2} + \EEE[\xi(Z_{1})]^{2}}\|_{\MQ,2},\FFF^{\Theta,p},L^{2}(\MQ)\big)\big)}~d\eta\\
&\leq 
2^{p+2} J(\FFF^{\Theta},\xi,\delta/2^{p+2}) + 2~ \delta~ \sqrt{\ln\big(2^{p+4}/\delta\big)}\quad\mbox{for}~\delta\in ]0, 1[. 
\end{align*}
Now, the statement of Lemma \ref{auxiliary I} follows easily from the observation 
$$
J(\FFF^{\Theta,p},\xi_{p},\delta)\leq \sqrt{2\ln(2)}~\delta + \sqrt{2}\int_{0}^{\delta}\sup_{\MQ\in \cM_{\textrm{\tiny fin}}}\sqrt{\ln\big(N\big(\varepsilon~\|\xi_{p}\|_{\MQ,2},\FFF^{\Theta,p},L^{2}(\MQ)\big)\big)}~d\varepsilon
$$
for $\delta > 0$
\hfill$\Box$
%%%%%%%%%%%%%%%%%%%%%%%%%%%%%%%%%%%%%%%%%%%%%%%%%%%%%%%%%%%%%%%%%%%%%%%%%%%%%%%%%
%%%%%%%%%%%%%%%%%%%%%%%%%%%%%%%%%%%%%%%%%%%%%%%%%%%%%%%%%%%%%%%%%%%%%%%%%%%%%%%%%
\subsection{Proof of results from Section \ref{generalized divergence risk measures}}
\label{Beweis compactification}
Let us introduce the sequence 
$\big(X_{n}\big)_{n\in\NNN}$ of random processes  
$$
X_{n}:\Omega\times\Theta\times\RRR\rightarrow\RRR,~X_{n}(\omega,\theta,x) := \frac{1}{n}~\sum_{j=1}^{n}\Big(\Phi^{*}\big(G(\theta,Z_{j}(\omega)) + x\big) - x\Big)\quad(n\in\NNN),
$$
and, under (A 2'), the mapping
$$
\psi_{\Phi}:\theta\times\RRR,~(\theta,x)\mapsto \EEE\Big[\Phi^{*}\big(G(\theta,Z_{1}) + x\big) - x\Big].
$$
The key for proving Proposition \ref{compactification} is the following 
observation.
\begin{lemma}
\label{key compactification}
Let (A 1) and (A 2') be fulfilled. Furthermore let $x_{0} > 1$ be from the effective domain of $\Phi$. Then with $\xi$ from (A 2') the following inequalities hold for $\theta\in\Theta, x\in\RRR$ and $n\in\NNN$.
\begin{align*}
& 
X_{n}(\cdot,\theta,x)
\geq
\max\Big\{-\Phi(0) - x, -\frac{x_{0}}{n}\sum_{j=1}^{n}\xi(Z_{j}) - \Phi(x_{0}) + x (x_{0} - 1)\Big\}\\
&
\psi_{\Phi}(\theta,x)
%&\geq& 
%\max\left\{-\Phi(0) - x, x_{0} \EEE[G(\theta,Z_{1})] - \Phi(x_{0}) + x (x_{0} - 1) \right\}\\ 
%&\geq&
\geq
\max\left\{-\Phi(0) - x, -x_{0} \EEE[\xi(Z_{1})] - \Phi(x_{0}) + x (x_{0} - 1) \right\}.
\end{align*}
In particular, $\psi_{\Phi}$ is bounded from below and also the path $X_{n}(\omega,\cdot,\cdot)$ for every $n\in\NNN$ and any $\omega\in\Omega$.
\end{lemma}
\begin{proof}
The inequalities $\Phi^{*}(y) \geq -\Phi(0)$ and $\Phi^{*}(y)\geq y x_{0} - \Phi(x_{0})$ hold for $y\in\RRR$ by definition of $\Phi^{*}$. Then, the inequalities in the statement follow easily. Next, notice that $\varphi(x) := \max\left\{-\Phi(0) - x, -x_{0} \EEE[\xi(Z_{1})] - \Phi(x_{0}) + x (x_{0} - 1) \right\}$ defines a continuous mapping 
$\varphi:\RRR\rightarrow\RRR$ which tends to $\infty$ for $x\to -\infty$ and $x\to\infty$.
%additionally satisfying
%$$
%\lim_{x\to - \infty}\varphi(x) = \lim_{x\to \infty}\varphi(x) = \infty.
%$$ 
Hence $\varphi$ is bounded from below, and thus also $\psi_{\Phi}$. In the same way it may be shown that $X_{n}(\omega,\cdot,\cdot)$ is bounded from below for $n\in\NNN$ and $\omega\in\Omega$. This completes the proof.
\end{proof}
In the next step we want to show that with high probability we 
%We may conclude from Lemma \ref{normal integrands} that the processes $X_{n}$ are normal integrand which have random minimizer. More importantly, it may be shown under the additional conditions (A 2), (A 3) that asymptotically we may 
may restrict simultaneously the minimizations of $\psi_{\Phi}$ and the processes $X_{n}$ to a compact subset of $\Theta\times\RRR$. 
More precisely, let us introduce the sets
\begin{align*}
& 
S(\psi_{\Phi}) := \Big\{(\theta,x)\in\Theta\times\RRR\mid \psi_{\Phi}(\theta,x) = \inf_{\theta\in\Theta\atop x\in\RRR}\psi_{\Phi}(\theta,x)\Big\},\\
&
S_{n}(\omega) := \Big\{(\theta,x)\in\Theta\times\RRR\mid X_{n}(\omega,\theta,x) = \inf_{\theta\in\Theta\atop x\in\RRR}X_{n}(\omega,\theta,x)\Big\}\quad(n\in\NNN,~\omega\in\Omega).
\end{align*}
\begin{theorem}
\label{stochastic equicontinuity}
Let (A 1), (A 2') be fulfilled. If $G(\cdot,z)$ is lower semicontinuous for $z\in\RRR^{d}$, then $S_{n}(\omega)$ is nonvoid for $n\in\NNN$ and $\omega\in\Omega$. Moreover, 
$$
S_{n}(\omega)\subseteq \Theta\times \big[x_{l}(x_{0},\xi,\delta),x_{u}(x_{0},\xi,\delta)\big]\quad\mbox{for}~n\in\NNN,~\delta > 0, \omega\in A_{n,\delta}^{\xi},
$$
where $x_{l}(x_{0},\xi,\delta),x_{u}(x_{0},\xi,\delta)$ are defined by \eqref{Eingrenzung1} and \eqref{Eingrenzung2} respectively, and $A_{n,\delta}^{\xi}\in\cF$ is as in the display of Proposition \ref{compactification}.
\end{theorem}

\begin{proof}
Since $\Phi^{*}$ is nondecreasing, we may observe by (A 2')
%\begin{equation}
%\label{reinziehen Absolutzeichen}
%|\Phi^{*}(y)|\leq\Phi^{*}(|y|)\quad\mbox{for every}~y\in\RRR.
%\end{equation} 
%Hence by (A 2)
\begin{eqnarray*}
\label{2Stern}
\sup_{\theta\in\Theta}\inf_{x\in\RRR}X_{n}(\cdot,\theta,x)\leq \sup_{\theta\in\Theta}\frac{1}{n}~\sum_{j=1}^{n}\Phi^{*}\big(G(\theta,Z_{j})\big)\leq \frac{1}{n}~\sum_{j=1}^{n}\Phi^{*}\big(\xi(Z_{j})\big).
\end{eqnarray*}
Then in view of Lemma \ref{key compactification}, we obtain for every $\omega\in\Omega$
$$
\inf_{\theta\in\Theta}X_{n}(\omega,\theta,x)
%\geq -\Phi(0) - x 
> \sup_{\theta\in\Theta}\inf_{x\in\RRR}X_{n}(\cdot,\theta,x)\quad\mbox{for}~
x \in\RRR\setminus [a_{n}(\omega),b_{n}(\omega)],
%< -\Phi(0) - \frac{1}{n}~\sum_{i=1}^{n}\xi_{2}\big(Z_{i}(\omega)\big),
$$
%and 
%$$
%\inf_{\theta\in\Theta}X_{n}(\omega,\theta,x)\geq  -\frac{x_{0}}{n}\sum_{i=1}^{n}\xi_{1}(Z_{i}) - \Phi(x_{0}) + x (x_{0} - 1) > \sup_{\theta\in\Theta}\inf_{x\in\RRR}X_{n}(\cdot,\theta,x)
%$$
%whenever
%$$
%x > {\Phi(x_{0}) + \frac{1}{n}\sum_{i=1}^{n}\xi_{2}\big(Z_{i}(\omega)\big) + \frac{x_{0}}{n}\sum_{i=1}^{n}\xi_{1}\big(Z_{i}(\omega)\big)\over x_{0} - 1}.
%$$
where 
$a_{n}(\omega) := -\Phi(0) - \frac{1}{n}~\sum_{j=1}^{n}\Phi^{*}\big(\xi\big(Z_{j}(\omega)\big)\big)$,
and
$$
b_{n}(\omega) := {\Phi(x_{0}) + \frac{1}{n}\sum_{j=1}^{n}\Phi^{*}\Big(\xi\big(Z_{j}(\omega)\big)\Big) + \frac{x_{0}}{n}\sum_{j=1}^{n}\xi\big(Z_{j}(\omega)\big)\over x_{0} - 1}.
$$
This means 
\begin{equation}
\label{null}
\inf_{\theta\in\Theta\atop x\in\RRR}X_{n}(\omega,\theta,x) = \hspace*{-0.5cm}\inf_{\theta\in\Theta\atop x\in [a_{n}(\omega),b_{n}(\omega)]}\hspace*{-0.5cm}X_{n}(\omega,\theta,x)\quad\mbox{and}\quad
S_{n}(\omega) \subseteq \Theta\times [a_{n}(\omega),b_{n}(\omega)]
\end{equation}
for $n\in\NNN, \omega\in\Omega$. Since $G(\cdot,z)$ is lower semicontinuous for $z\in\RRR^{d}$, and since $\Phi^{*}$ is nondecreasing as well as continuous, the mapping $X_{n}(\omega,\cdot,\cdot)$ is lower semicontinuous on the compact set $\Theta\times[a_{n}(\omega),b_{n}(\omega)]$ for $\omega\in\Omega$. As a consequence $S_{n}(\omega)$ is nonvoid for $n\in\NNN$ and $\omega\in\Omega$. We may also conclude from (\ref{null}) that 
$S_{n}(\omega)$ is contained in the set $\Theta\times \big[x_{l}(x_{0},\xi,\delta), x_{u}(x_{0},\xi,\delta)\big]$ if $\omega\in A_{n,\delta}^{\xi}$. The proof is complete.
%Finally by assumption, the sequences $\big(\xi_{1}(Z_{i})\big)_{i\in\NNN}$ and  $\big(\xi_{2}(Z_{i})\big)_{i\in\NNN}$ consist of independent integrable random variables which are identically distributed as $\xi_{1}(Z_{1})$ and $\xi_{2}(Z_{1})$ respectively. Then the classical strong law of large numbers enables us to find some $A\in\cF$ with $\MP(A) = 1$ such that 
%$$
%\lim_{n\to\infty}\frac{1}{n}\sum_{i=1}^{n}\xi_{j}\big(Z_{i}(\omega)\big) = \EEE[\xi_{j}(Z_{1})]\quad\mbox{for}~j\in\{1,2\}~\mbox{and}~\omega\in A.  
%$$
%Fix now any $\omega\in A$. We may select $n(\omega)\in\NNN$ such that $\omega\in A_{n}$ for $n\in\NNN$ with $n\geq n(\omega)$. Hence the remaining part of Theorem \ref{stochastic equicontinuity} may be concluded immediately. The proof is complete.
\end{proof}
%Using the notations introduced before Theorem \ref{stochastic equicontinuity}, 
We 
may also derive compactness of the set $S(\psi_{\Phi})$ of minimizers of $\psi_{\Phi}$.
\begin{lemma}
\label{basic observations}
Let (A 1), (A 2') be fulfilled, and let $G(\cdot,z)$ be lower semicontinuous for $z\in\RRR^{d}$. Then the mapping $\psi_{\Phi}$ is lower semicontinuous, and the set $S(\psi_{\Phi})$ is nonvoid and compact, satisfying 
$$
S(\psi_{\Phi})\subseteq \Theta\times \big[x_{l}(x_{0},\xi,\delta), x_{u}(x_{0},\xi,\delta)\big]\quad\mbox{for}~\delta > 0.
$$
\end{lemma}
\begin{proof}
%By definition of $\Phi^{*}$ and (A 2)
%\begin{eqnarray*}
%f(\theta,x)
%\geq
%\max\left\{-\Phi(0) - x, -x_{0} \EEE[\xi_{1}(Z_{1})] - \Phi(x_{0}) + x (x_{0} - 1) \right\}\quad\mbox{for}~\theta\in\Theta, x\in\RRR.
%\end{eqnarray*}
First of all by (A 2') along with monotonicity of $\Phi^{*}$ we may observe
$$
\sup_{\theta}\inf_{x\in\RRR}\psi_{\Phi}(\theta,x)\leq\sup_{\theta\in\Theta}\EEE\big[\Phi^{*}\big(G(\theta,Z_{1})\big)\big]\leq\EEE\big[\Phi^{*}\big(\xi(Z_{1})\big)\big].
$$
Then in view of Lemma \ref{key compactification} we may conclude that $\psi_{\Phi}(\theta,x) > \inf \psi_{\Phi}$ if 
$$
x < -\Phi(0) - \EEE\big[\Phi^{*}\big(\xi(Z_{1})\big)\big]\quad\mbox{or}\quad x >  {\EEE\big[\Phi^{*}\big(\xi(Z_{1})\big] + x_{0} \EEE\big[\xi(Z_{1})\big] + \Phi(x_{0}) \over x_{0} - 1}.
$$
Hence $\psi_{\Phi}$ and its restriction to $\big[x_{l}(x_{0},\xi,\delta), x_{u}(x_{0},\xi,\delta)\big]$ have the same infimal value, and $S(\psi_{\Phi}) \subseteq \Theta\times \big[x_{l}(x_{0},\xi,\delta), x_{u}(x_{0},\xi,\delta)\big]$ for $\delta > 0$.
\medskip

Lower semicontinuity of $G$ in $\theta$ implies that $(\theta,x)\mapsto\Phi^{*}\big(G(\theta,Z_{1}(\omega)) + x\big) + x$ is a lower semicontinuous mapping on $\Theta\times\RRR$ for any $\omega\in\Omega$ because $\Phi^{*}$ is nondecreasing and continuous. In addition by definition of $\Phi^{*}$ we obtain for any $\eta > 0$ and $\omega\in\Omega$
$$
\inf_{\theta\in\Theta\atop |x|\leq\eta}\Big(\Phi^{*}\big(G(\theta,Z_{1}(\omega)) + x\big) - x\Big)
\geq 
\inf_{|x|\leq\eta}\big(-\Phi(0) - x\big)\geq -\Phi(0) - \eta.
$$
Then an easy excercise of Fatou's Lemma shows that $\psi_{\Phi}$ is lower semicontinuous. Hence by compactness of $\Theta\times \big[x_{l}(x_{0},\xi,1), x_{u}(x_{0},\xi_{1},\xi,1)\big]$ the set $S(\psi_{\Phi})$ is a nonvoid compact subset of $\RRR^{m + 1}$. This completes the proof.
\end{proof}
Now we are ready to show Proposition \ref{compactification}.
\medskip

\noindent
\textbf{Proof of Proposition \ref{compactification}:}\\
Recall the representation of the genuine optimization problem \eqref{optimization general II} via Theorem \ref{optimized certainty equivalent}, and the representation of the problem associated with the SAA by \eqref{optimization approximativ}. Then the entire statement of Proposition \ref{compactification} may be derived easily from Theorem \ref{stochastic equicontinuity} along with Lemma \ref{basic observations}.
\hfill$\Box$
\medskip

Let us turn over to the proof of Lemma \ref{relationship}.
\medskip

\noindent
\textbf{Proof of Lemma \ref{relationship}:}\\[0.1cm]
Since $\Phi^{*}$ is convex, its right-sided derivative $\Phi^{*'}_{+}$ is nondecreasing. Then the inequality $|\Phi^{*}(x) - \Phi^{*}(y)|\leq \Phi^{*'}_{+}(x\vee y) |x - y|$ holds for $x,y\in\RRR$. In particular this yields $|\Phi^{*}(x)|\leq \Phi^{*'}_{+}(x^{+}) |x|$  for $x\in\RRR$ because $\Phi^{*}(0) = 0$. Hence we may observe 
$$
\big|G_{\Phi}\big((\theta,x),z\big)\big|\leq [\Phi^{*'}_{+}(\xi + \sup I) + 1](\xi(z) + \sup I)\leq  
C_{\FFF^{\Theta}_{\Phi,I}}(z)\quad\mbox{for}~(\theta,x)\in\Theta\times I, z\in\RRR^{d}
$$
and 
\begin{align*}
&
\big|G_{\Phi}\big((\theta,x),z\big) - G_{\Phi}\big((\vartheta,y),z\big)\big|^{2}\\
&\leq 
4 [\Phi^{*'}(\xi + \sup I) + 1]^{2} |G(\theta,z) - G(\vartheta,z)|^{2} + 4 [\Phi^{*'}(\xi + \sup I) + 1]^{2} |x - y|^{2}
\end{align*}
for $(\theta,x), (\vartheta,y)\in\Theta\times I$ and $z\in\RRR^{d}$. So firstly, $C_{\FFF^{\Theta}_{\Phi,I}}$ is a positive envelope of $\FFF^{\Theta}_{\Phi,I}$. Secondly, we may invoke Theorem 2.10.20 from \cite{vanderVaartWellner1996} to conclude 
\begin{align*}
&J(\FFF^{\Theta}_{\Phi,I},C_{\FFF^{\Theta}_{\Phi,I}},\delta)\\
&\leq \sqrt{2\ln(2)}~\delta + \sqrt{2}~\int_{0}^{\delta}\sup_{\MQ\in \cM_{\textrm{\tiny fin}}}\sqrt{\ln\big(N\big(\varepsilon~\|C_{\FFF^{\Theta}_{\Phi,I}}\|_{\MQ,2},\FFF^{\Theta}_{\Phi,I},L^{2}(\MQ)\big)\big)}~d\varepsilon\\
&\leq 
\sqrt{2\ln(2)}~\delta +\sqrt{2}~J(\FFF^{\Theta},\xi,\delta) + 
\sqrt{2} \int_{0}^{\delta}\sqrt{\ln\big(N (\eta \sup I,I,|\cdot|)\big)}~d\eta\quad\mbox{for}~\delta > 0,
\end{align*}
where $N (\eta\cdot\sup I,I,|\cdot|)$ denotes the minimal number to cover $I$ by intervals of the form $[x_{i}-\eta\cdot\sup I,x_{i} + \eta\cdot\sup I]$ with $x_{i}\in I$. Since 
$N(\eta\cdot\sup I,I,|\cdot|)\leq (\sup I - \inf I)/(\eta\cdot \sup I)$ holds for $\eta > 0$, and since $\sup I - \inf I \leq 2\sup I$, we obtain via the change of variable formula  
\begin{align*}
\int_{0}^{\delta}\sqrt{\ln\big(N (\eta\cdot\sup I,I,|\cdot|)\big)}~d\eta 
%&
\leq 
\delta \int_{0}^{1}\sqrt{\ln\big((2/\delta)/\varepsilon\big)}~d\varepsilon
%\\
%&\leq 
%\sqrt{2} \delta \int_{0}^{1}\sqrt{\ln\big((1/\delta)/\varepsilon\big)}~d\varepsilon
\quad\mbox{for}~\delta > 0.
\end{align*}
Now, we may  finish the proof by applying \eqref{Integralabschaetzung} for every $\delta\in ]0,\exp(-1)]$. 
\hfill$\Box$
\medskip

\noindent
\textbf{Proof of Lemma \ref{countable parameter subsets}:}\\[0.1cm]
For $n\in\NNN$, $\theta\in\Theta$, $x\in I$ and $(z_{1},\ldots,z_{n})\in\RRR^{d n}\setminus N_{n}$ we may conclude immediately from (A 3'') along with continuity of $\Phi^{*}$
$$
\inf_{(\vartheta,y)\in\overline{\Theta}\times I\cap\mathbb{Q}}\max_{j\in\{1,\ldots,n\}}\big|G_{\Phi}\big((\theta,y),z_{j}\big) - G_{\Phi}\big((\vartheta,x),z_{j}\big)\big| = 0.
$$
Next we may find by (A 3'') for a fixed $(\theta,x)\in\Theta\times I$ some sequence $\big(\theta_{n},x_{n}\big)_{n\in\NNN}$ in $\overline{\Theta}\times I\cap\mathbb{Q}$ such 
that $\EEE\big[|G(\theta_{n},Z_{1}) - G(\theta,Z_{1})|\big]\to 0$ and $x_{n}\to x$. In particular $G_{\Phi}\big((\theta_{n},x_{n}),Z_{1}\big)\to G_{\Phi}\big((\theta,x),Z_{1}\big)$ in probability because $\Phi^{*}$ is continuous. Since $\Phi^{*}$ is convex, nondecreasing with $\Phi^{*}(0) = 0$, we may observe $|\Phi^{*}(y)|\leq \Phi^{*}(|y|)$ for $y\in\RRR$.
Hence by (A 2') along with monotonicity of $\Phi^{*}$ we have for $\xi$ from (A 2')
$$
\sup_{n\in\NNN}\big|G_{\Phi}\big((\theta_{n},x_{n}),Z_{1}\big)\big|\leq \Phi^{*}\Big(\xi(Z_{1}) + \sup_{n\in\NNN}|x_{n}|\Big) + \sup_{n\in\NNN}|x_{n}|.
$$
Hence by (A 2') again the random variables $G_{\Phi}\big((\theta_{n},x_{n}),Z_{1}\big)$ are dominated by some integrable random variable. Then an application of Vitalis' theorem 
(see e.g. \cite[Theorem 21.4]{Bauer2001}) yields $\EEE\big[\big|G_{\Phi}\big((\theta_{n},x_{n}),Z_{1}\big) - G_{\Phi}\big((\theta,x),Z_{1}\big)\big|\big]\to 0$. This completes the proof. 
\hfill$\Box$

%%%%%%%%%%%%%%%%%%%%%%%%%%%%%%%%%%%%%%%%%%%%%%%%%%%%%%%%%%%%%%%%%%%%%%%%%%%%%%%%%%%%%%%%%%%%%%%%%%%%%%%%%%%%%%%%%%%%%%%%%%%%%%%%%%%%%%%%%%%%%%%%%%%%%%%%%%%%%%%%%%%%%%%%%%%%%%%%%
%\bibliographystyle{plainnat}
%\bibliography{SAA-risk-averse}
%\bibliographystyle{siamplain}
%\bibliography{references}

\end{document}